\documentclass{aims-arxiv}
\usepackage{amsmath}
\usepackage{paralist}
\usepackage[export]{adjustbox}

\usepackage{epsfig} 
\usepackage{graphicx}
\usepackage{caption}
\usepackage{multicol}
\usepackage{epstopdf}
\usepackage{algorithmic}
\usepackage{subcaption}
\usepackage[pagewise]{lineno}
\usepackage{pgfplots}
\usepackage{graphicx}  \usepackage{epstopdf}
\usepackage[colorlinks=true]{hyperref}
\hypersetup{urlcolor=blue, citecolor=red}

\def\currentvolume{}
\def\currentissue{}
\def\currentyear{}
\def\currentmonth{}
\def\ppages{}

\newtheorem{theorem}{Theorem}[section]

\newtheorem{lemma}[theorem]{Lemma}

\theoremstyle{definition}
\newtheorem{definition}[theorem]{Definition}
\newtheorem{remark}{Remark}

\title[On a pore-scale stationary diffusion equation] 
{On a pore-scale stationary diffusion equation: scaling effects and correctors for the homogenization limit}

\author[V. A. Khoa and T. K. Thoa Thieu and E. R. Ijioma]{}

\subjclass{Primary: 35B27 , 35C20 , 35D30 , 65M15.}
\keywords{Pore-scale model; nonlinear elliptic equations; perforated domains; linearization; asymptotic analysis; corrector estimates.}

\email{vakhoa.hcmus@gmail.com}
\email{thikimthoa.thieu@gssi.it}
\email{e.r.ijioma@gmail.com}

\thanks{The work of V. A. K was partly supported by the Research Foundation-Flanders (FWO) under the project named ``Approximations for forward and inverse reaction-diffusion problems related to cancer models''. This work was also supported by US Army Research Laboratory and US Army Research Office grant W911NF-19-1-0044.}

\thanks{$^*$Corresponding author: Vo Anh Khoa.}

\begin{document}
	\maketitle
	
	\centerline{\scshape Vo Anh Khoa$^*$}
	\medskip
	{\footnotesize
		\centerline{Faculty of Sciences, Hasselt University, Campus Diepenbeek, BE3590 Diepenbeek, Belgium}
		\centerline{Department of Mathematics and Statistics, University of North Carolina at Charlotte, }
		\centerline{Charlotte, North Carolina 28223, USA}
	} 
	
	\medskip
	
	\centerline{\scshape  Thi Kim Thoa Thieu}
	\medskip
	{\footnotesize
		\centerline{Department of Mathematics, Gran Sasso Science Institute, Viale Francesco Crispi 7, }
		\centerline{L'Aquila 67100, Italy}
		\centerline{Department of Mathematics and Computer Science, Karlstad University,}
		\centerline{Universitetsgatan 2, Karlstad, Sweden}
	}
	\medskip  
	\centerline{\scshape  Ekeoma Rowland Ijioma}
	\medskip
	{\footnotesize
		\centerline{Meiji Institute for Advanced Study of Mathematical Sciences,}
		\centerline{4-21-1 Nakano, Nakano-ku, Tokyo, Japan}
	}
	
	\bigskip
	
	\centerline{(Communicated by the associate editor name)}

	\begin{abstract}
		In this paper, we consider a microscopic semilinear elliptic equation posed in periodically perforated domains and associated with the Fourier-type condition on internal micro-surfaces. The first contribution of this work is the construction of a reliable linearization scheme that allows us, by a suitable choice of scaling arguments and stabilization constants, to prove the weak solvability of the microscopic model. Asymptotic behaviors of the microscopic solution with respect to the microscale parameter are thoroughly investigated in the second theme, based upon several cases of scaling. In particular, the variable scaling illuminates the trivial and non-trivial limits at the macroscale, confirmed by certain rates of convergence.  
		Relying on classical results for homogenization of multiscale elliptic problems, we design a modified two-scale asymptotic expansion to derive the corresponding macroscopic equation, when the scaling choices are compatible. Moreover, we prove the high-order corrector estimates for the homogenization limit in the energy space $H^1$, using a large amount of energy-like estimates. A numerical example is provided to corroborate the asymptotic analysis.
	\end{abstract}
	
	\section{Introduction}
	\subsection{Background and statement of the problem}
	
	We assume that a porous medium $\Omega^{\varepsilon} \subset \mathbb{R}^d$ ($d\in\mathbb{N}^{*}$) possesses a uniformly
	periodic microstructure whose length scale is defined by a small parameter
	(microscale parameter) $0<\varepsilon\ll1$. In practice, $\varepsilon$ is defined as the ratio of the characteristic length of the microstructure to a characteristic macroscopic length. In this paper, the porous medium of interest
	contains a large amount of very small holes and thus can be viewed
	as a perforated domain. Largely inspired by \cite{KMK15}, our work aims at understanding the spread of concentration of colloidal particles $u_{\varepsilon}:\Omega^{\varepsilon}\to\mathbb{R}$
	in a saturated porous tissue $\Omega^{\varepsilon}$ with a cubic 
	cell $Y=\left[0,1\right]^{d}$. This kind of tissues can be illustrated in Figure \ref{fig:1} as a schematic representation of a natural soil. Since the constitutive properties
	of the microstructure repeat periodically, the molecular diffusion
	coefficient $\mathbf{A}:Y\to\mathbb{R}^{d\times d}$ is assumed to
	vary in the cell or in a material point $x\in\Omega^{\varepsilon}$. Accordingly, it can be expressed as $\mathbf{A}\left(x/\varepsilon\right)$. We consider the presence of a volume reaction $\mathcal{R}:\mathbb{R}\to\mathbb{R}$
	combined with an internal source $f:\Omega^{\varepsilon}\to\mathbb{R}$.
	Moreover, we also consider a chemical reaction $\mathcal{S}:\mathbb{R}\to\mathbb{R}$
	for the immobile species along with deposition coefficients at the
	internal boundaries, denoted by $\Gamma^{\varepsilon}$. On the other
	hand, the colloidal species stay constant at the exterior boundary,
	denoted by $\Gamma^{\text{ext}}$. Mathematically, the governing equations describing this process is given by
	\[
	\left(P_{\varepsilon}\right):\;\begin{cases}
	\nabla\cdot\left(-\mathbf{A}\left(x/\varepsilon\right)\nabla u_{\varepsilon}\right)+\varepsilon^{\alpha}\mathcal{R}\left(u_{\varepsilon}\right)=f(x) & \text{in}\;\Omega^{\varepsilon},\\
	-\mathbf{A}\left(x/\varepsilon\right)\nabla u_{\varepsilon}\cdot\text{n}=\varepsilon^{\beta}\mathcal{S}\left(u_{\varepsilon}\right) & \text{across}\;\Gamma^{\varepsilon},\\
	u_{\varepsilon}=0 & \text{across}\;\Gamma^{\text{ext}}.
	\end{cases}
	\]\label{pt_P}
	
	\begin{figure}[htbp]
		\centering
		\includegraphics[scale = 0.3]{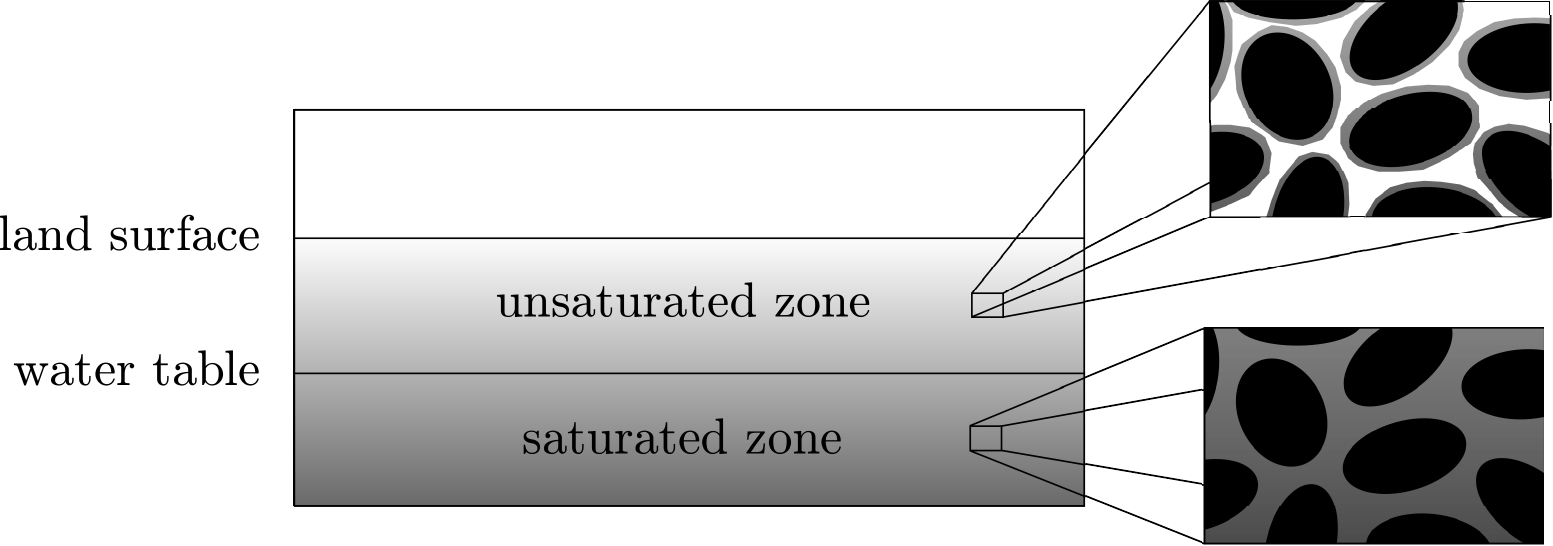}
		\caption{\small A schematic representation of a natural soil. The figure is followed from \cite{ray2013thesis}.}
		\label{fig:1}
	\end{figure}
	
	Typically, the prototypical problem $({P}_{\varepsilon})$ forms the standard model for diffusion, aggregation and surface deposition of a concentration in a porous and highly
	heterogeneous medium; cf. \cite{HJ91,KAM14,KMK15}. Starting from the heat conduction problem in composite materials with inclusions (cf. \cite{CP99}), much of the current literature on the analysis of this elliptic model pays meticulous attention to different contexts involving rigorous mathematical treatments, physical modelling (e.g., \cite{DHS16,SPPK12,Gaudiello2018}) and numerical standpoints (e.g., \cite{Sarkis}). Especially, understanding the asymptotic behaviors of solution of $({P}_{\varepsilon})$ is essential in the studies of the classical spectral problems as investigated in, e.g., \cite{CM02,Mel95}.
	
	\subsection{Main goals}
	
	It is well known that due to
	the fast oscillation in the diffusion coefficient $\mathbf{A}\left(x/\varepsilon\right)$, the number of mesh-points at any discretization level is of the order
	$\varepsilon^{-d}$, which consequently reveals a huge complexity
	of computations as
	$\varepsilon\searrow 0^{+}$. Therefore, when dealing with such multiscale problems,
	one usually targets the upscaled models where the oscillation is no longer involved, using different techniques of asymptotic analysis with respect to $\varepsilon$. In this work, our main finding is to point out the trivial and non-trivial macroscopic models, depending on every single case of $\alpha$ and $\beta$. Note that the major mathematical challenge we meet here is that $\alpha, \beta$ are real constants. We mention some physical implications of these parameters, although these variable scalings stem from our mathematical interest. The case $\alpha < 0$ and $\beta >0$ implies that the volume reaction is large, while the surface reaction of the concentration is slow. Thus, $u_{\varepsilon}$ can be expected to be slowly changing on $\Gamma^{\varepsilon}$. When $\alpha < 0$ and $\beta >0$, we have a dominant surface reaction and the volume reaction is
	negligible. Meanwhile, on $\Gamma^{\varepsilon}$, this means that  $u_{\varepsilon}$ is rapidly changing. As to the literature on this topic, the reader can be referred to \cite{RMK12,ray2013thesis,Khoa2019,Frank2011,Ray2012}, where the variable scaling has been considered earlier in complex diffusion problems of charged colloidal particles. Besides, the case $\alpha > 0$ can be related to the
	context of low-cost control problems on perforated domains in \cite{Muthukumar2009}. 
	
	The variable scaling not only requires a careful adaptation of classical homogenization results for elliptic problems, but also needs a particular investigation into the macroscopic problems obtained after the homogenization limit. When the limit function is non-trivial, we particularly design a new asymptotic expansion, which is needed by the presence of the scaling parameters. Aside from the derivation of the macroscopic problem, we delve into having the high-order corrector estimates, driven by a large amount of energy estimates. In the same manner, rates of convergence to the trivial limit function are under scrutiny. As the inception stage of this asymptotic analysis, we only focus on the speed of convergence when $\varepsilon\searrow 0^{+}$, while the regularity assumptions may not be minimal.
	
	The high-order corrector estimates we prove in this paper are involving the presence of the scaling parameters. This is the extended follow-up result of our earlier works \cite{KM16,KM16-1,Khoa17}, where we wish to estimate the differences between micro--macro concentrations and micro--macro concentration gradients in the energy space of perforated domains. It is also in the same line with the theoretical findings in \cite{Kim2018,CP99-1,Armstrong2016,Allaire1999,Onofrei2007,Onofrei2012,Griso2004}. Our preceding works show that the macroscopic problem can be self-linear, albeit the  semilinear microscopic problem. We find that this is caused by the scaled structure of the involved nonlinear reaction, which sometimes leads to the self-iterative auxiliary problems. However, there are somehow the cases that the auxiliary problems remain semilinear and thus, the fixed-point homogenization argument is required.
	
	Another novelty we present here is the exploration of a linearization scheme for $({P}_{\varepsilon})$ under a scaled Hilbert space. This scheme not only proves the weak solvability of $({P}_{\varepsilon})$, but also provides an insight to expect the asymptotic analysis, due to the \emph{a priori} estimates we obtain after the iteration limit. As far as the linearization-based algorithm is concerned, it has been profoundly developed for a long time in numerical methods for nonlinear PDEs. It is worth mentioning that the J\"ager--Ka\v cur scheme (see, e.g., \cite{Kacur99}) was investigated as the very first contribution in this branch. It plays a vital role in solving some classes of one-dimensional parabolic problems, but it is not really effective in high dimensions. Using the same idea, Long et al. \cite{LDD02} rigorously proved the local existence and uniqueness of a weak solution of a Kirchoff--Carrier wave equation in one-dimension.  We also recall the linearization by the monotonicity of iterations, for example, introduced in the monograph \cite{Pao93} involving the concepts of \emph{sub}- and \emph{super-solution}. However, its drawback comes from the way the initial loop is chosen, which must be far away from the true solution, whilst in general it can be taken by the already known initial or boundary information.
	
	
	
	\subsection{Outline}
	Our paper is structured as follows. In Section \ref{sec:2}, we provide the essential notation and working assumptions on data used in the analysis. In Section \ref{sec:3}, we design a linearization scheme to prove the weak solvability of the microscopic model, where the main result in this part is stated in Theorem \ref{mainthm-2}. In Section \ref{sec:4}, we design several two-scale asymptotic expansions, corresponding to some particular microscopic models contained in $(P_{\varepsilon})$. Accordingly, we obtain distinctive convergence results. We conclude this paper by some numerical discussions included in Section \ref{sec:4.3}.
	

	\section{Preliminaries}\label{sec:2}
	\subsection{Geometrical description of a perforated medium}
	
	Let $\Omega$ be a bounded, open and connected domain in $\mathbb{R}^d$ with a Lipschitz boundary. Typically, we can consider it as a reservoir in three dimensions. Now, let $Y$ be the unit cell defined by
	\[
	Y:=\left\{ \sum_{i=1}^{d}\lambda_{i}\vec{e}_{i}:0<\lambda_{i}<1\right\} ,
	\]
	where $\vec{e}_{i}$ denotes the $i$th unit vector in $\mathbb{R}^{d}$. In addition, we assume that this cell is made up of two open sets: $Y_l$ -- the liquid part and $Y_s$ -- the solid part which is impermeable to solute concentrations satisfy $\bar{Y}_{l}\cup\bar{Y}_{s}=\bar{Y}$ and $Y_{l}\cap Y_{s}=\emptyset,\bar{Y}_{l}\cap\bar{Y}_{s}=\Gamma$ possessing a non-zero $\left(d-1\right)$-dimensional measure. On the other hand,  suppose that the solid part $Y_s$ stays totally inside in the cell $Y$, i.e. it does not intersect the cell's boundary $\partial Y$. Consequently, the liquid part $Y_l$ is connected.
	
	Let $Z\subset\mathbb{R}^{d}$ be a hypercube. For $X\subset Z$ we
	denote by $X^{k}$ the shifted subset
	\[
	X^{k}:=X+\sum_{i=1}^{d}k_{i}\vec{e}_{i},
	\]
	where $k=\left(k_{1},...,k_{d}\right)\in\mathbb{Z}^{d}$ is a vector
	of indices.
	
	We scale this reservoir by a parameter $\varepsilon>0$ which represents the ratio of the cell size to the size of the whole reservoir. Often, this scale factor is small. We further assume that $\Omega$ is completely covered by a regular mesh consisting of three $\varepsilon$-scaled and shifted cells: the scaled liquid, solid parts and boundary. More precisely, the solid part is defined as the union of the cell regions $\varepsilon Y_{s}^{k}$, i.e.
	\[
	\Omega_{0}^{\varepsilon}:=\bigcup_{k\in\mathbb{Z}^{d}}\varepsilon Y_{s}^{k},
	\]
	while the liquid part is given by
	\[
	\Omega^{\varepsilon}:=\bigcup_{k\in\mathbb{Z}^{d}}\varepsilon Y_{l}^{k},
	\]
	and we denote the micro-surface by $\Gamma^{\varepsilon}:=\partial\Omega_{0}^{\varepsilon}$.
	
	Note that we now assume $\partial\Omega \equiv \Gamma^{\text{ext}}$ and our perforated domain $\Omega^{\varepsilon}$ is bounded, connected and possesses $C^{2}$-internal boundary.  We also denote throughout this paper $\mbox{n}:=\left(n_{1},...,n_{d}\right)$ as the unit outward normal vector on the boundary $\Gamma^{\varepsilon}$. In Figure \ref{fig:2}, we show an illustration of scales from a soil structure and the perforated domain with its unit cell.  The representation of the periodic geometries is in line with \cite{KM16, RMK12} and references therein.
	
	\begin{figure}
		\begin{centering}
			\includegraphics[scale = 0.3]{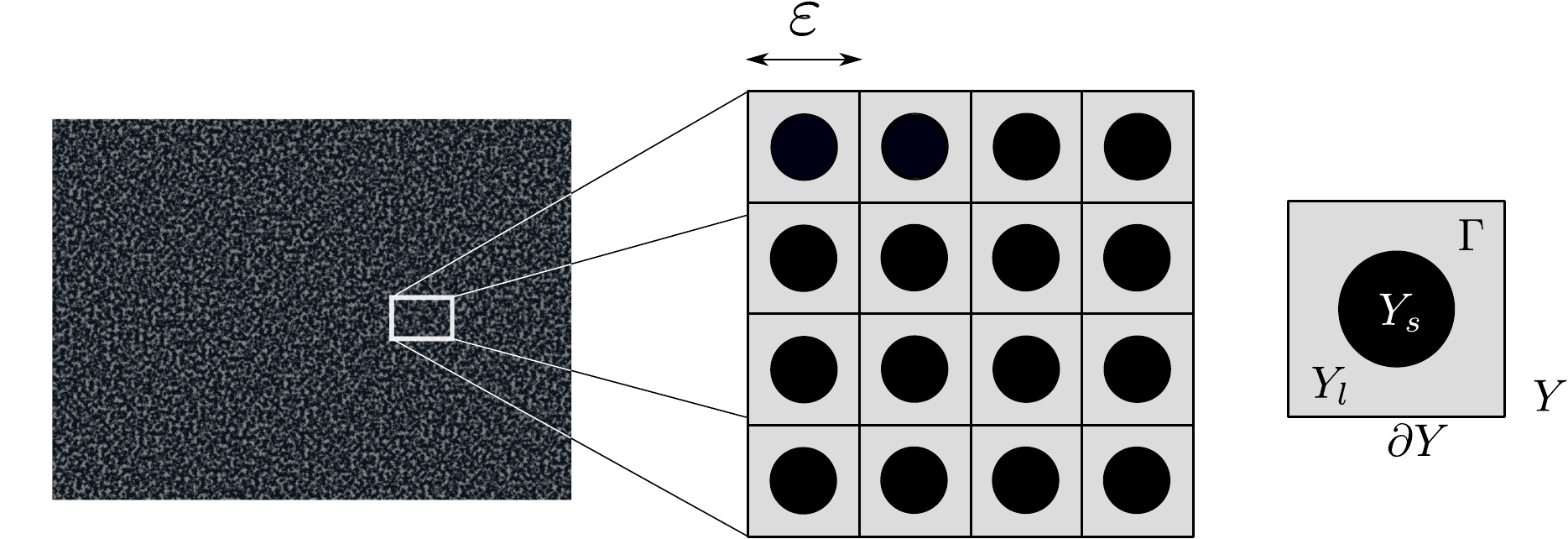}
			\par\end{centering}
		\caption{\small A schematic representation of the scaling procedure within a natural soil and the corresponding sample periodically perforated domain with its unit cell.}
		\label{fig:2}
	\end{figure}
	\subsection{Notation and assumptions on data}
	
	We denote by $x\in\Omega^{\varepsilon}$ the macroscopic variable and by $y=x/\varepsilon$ the microscopic variable representing fast variations at the microscopic geometry. With this convention, we write $\mathbf{A}(x/ \varepsilon) = \mathbf{A}_{\varepsilon}(x) = \mathbf{A}(y)$. Let us define the function space 
	$$V^\varepsilon := \{v \in H^1(\Omega^\varepsilon)| v = 0 \text{ on } \Gamma^{\text{ext}}\},$$
	which is a closed subspace of the Hilbert space $H^1(\Omega^\varepsilon)$, and thus endowed with the semi-norm
	$$\|v\|_{V^\varepsilon} := \left(\sum_{i = 1}^{d}\int_{\Omega^\varepsilon}\left|\frac{\partial v}{\partial x_i}\right|^2dx\right)^{1/2} \text{ for all } v \in V^{\varepsilon}.$$\\
	Obviously, this norm is equivalent (uniformly in $\varepsilon$) to the usual $H^1$-norm by the Poincar\'{e} inequality (cf. \cite[Lemma 2.1]{CP99}). 
	
	Let us define the function space $\mathcal{W}_\varepsilon := L^2(\Gamma^\varepsilon) \cap L^2(\Omega^\varepsilon)$ with the inner product
	\begin{align}
	\langle u,v\rangle_{\mathcal{W}_\varepsilon} := \langle u,v\rangle_{L^2(\Gamma^\varepsilon)} + \langle u,v\rangle_{L^2(\Omega^\varepsilon)} \qquad  \text{ for } u, v \in \mathcal{W}_\varepsilon, \nonumber
	\end{align}
	and with the corresponding norm $
	\|u\|_{{\mathcal{W}}_\varepsilon}^2 := \|u\|_{L^2(\Gamma^\varepsilon)}^2 + \|u\|_{L^2(\Omega^\varepsilon)}^2$. Then for each $\varepsilon > 0$ we introduce the function space $\widetilde{\mathcal{W}}_\varepsilon$ equipped with the following inner product
	\begin{align}
	\langle u,v\rangle_{\widetilde{\mathcal{W}}_\varepsilon} := \langle \nabla u, \nabla v\rangle_{L^2(\Omega^\varepsilon)} + (\varepsilon^\alpha + \varepsilon^\beta)\langle u,v\rangle_{\mathcal{W}_\varepsilon} \quad \text{ for } u,v \in \widetilde{\mathcal{W}}_\varepsilon,\nonumber
	\end{align}
	and the corresponding norm is given by
	$
	\|u\|_{\widetilde{\mathcal{W}}_\varepsilon}^2 := \|\nabla u\|_{L^2(\Omega^\varepsilon)}^2 + (\varepsilon^\alpha + \varepsilon^\beta)\|u\|_{\mathcal{W}_\varepsilon}^2$.
	This space can be considered as the intersection between $V^\varepsilon$ and the $\varepsilon$-scaled $\mathcal{W}_\varepsilon$. Hereby, $\widetilde{\mathcal{W}}_\varepsilon$ is a Hilbert space.  
	
	We introduce a bilinear form 
	$a : \widetilde{\mathcal{W}}_\varepsilon \times \widetilde{\mathcal{W}}_\varepsilon \to \mathbb{R}$ by \begin{align}\label{bilinearform}
	a(u,\varphi) := \int_{\Omega^\varepsilon}\mathbf{A}_\varepsilon(x)\nabla u\cdot\nabla \varphi dx.
	\end{align}
	
	To be successful with our analysis below, we need the following assumptions:
	\begin{itemize}
		\item [($\text{A}_1$)] The diffusion coefficient $\mathbf{A}(y) \in L^{\infty}(\mathbb{R}^d)$ is $Y$-periodic and symmetric. It satisfies the uniform ellipticity condition, i.e there exists positive constants  $\underline{\gamma}, \overline{\gamma}$ independent of $\varepsilon$ such that
		$$\underline{\gamma}\left|\xi\right|^2 \leq \mathbf{A}(y)\xi_i\xi_j \leq \overline{\gamma}\left|\xi\right|^2 \text{ for any } \xi \in \mathbb{R}^d.$$ 
		\item [($\text{A}_2$)] The reaction terms $\mathcal{S} :  \mathbb{R}\to \mathbb{R}$ and $\mathcal{R}: \mathbb{R} \to \mathbb{R}$ are Carath\'{e}odory functions and globally Lipschitz.
		\item[($\text{A}_3$)] $\mathcal{S}$ and $\mathcal{R}$ do not degenerate, i.e there exist positive constants $\delta_0$ and $\delta_1$ independent of $\varepsilon$ such that $0 < \delta_0 \leq \mathcal{S}', \mathcal{R}' \leq \delta_1$ a.e. in $\mathbb{R}$.
		\item[($\text{A}_4$)] The internal source $f$ is a smooth function in $\Omega$.
	\end{itemize}

	In the sequel, all the constants $C$ are independent of the scale factor $\varepsilon$, but their precise values may differ from line to line and may change even within a single chain of estimates.
	
	\section{Weak solvability of $(P_{\varepsilon})$}\label{sec:3}
	Obviously, solvability of microscopic problems is of great importance in mathematical analysis; see e.g. \cite{Oleinik1992,Papanicolau1978,Schmuck2013}. In this section, we design a linearization scheme in line with \cite{Slodika2001} to investigate the well-posedness of $(P_\varepsilon)$. To do so, we accordingly need to derive the weak formulation of $(P_\varepsilon)$. Multiplying $(P_\varepsilon)$ by a test function $\varphi \in \widetilde{\mathcal{W}}_\varepsilon$ and using Green's formula, we arrive at
	the following definition of a weak solution of $(P_\varepsilon)$.
	
	\begin{definition}\label{dn_weak}
		For each $\varepsilon > 0$, $u_\varepsilon \in \widetilde{\mathcal{W}}_\varepsilon$ is a weak solution to $(P_\varepsilon)$, provided that
		\begin{align}\label{pt4}
		a(u_\varepsilon,\varphi) + \varepsilon^{\beta}\langle \mathcal{S}(u_{\varepsilon}),\varphi\rangle_{L^2(\Gamma^{\varepsilon})} + \varepsilon^{\alpha}\langle \mathcal{R}(u_\varepsilon),\varphi\rangle_{L^2(\Omega^{\varepsilon})} = \langle f,\varphi \rangle_{L^2(\Omega^{\varepsilon})} \quad \text{ for all } \varphi \in \widetilde{\mathcal{W}}_\varepsilon.
		\end{align}
	\end{definition}
	Let us now introduce the definition of an approximation of \eqref{pt4}.
	
	\begin{definition}\label{dn_app}
		For each $\varepsilon > 0$, a linearization scheme for the weak formulation in Definition \ref{dn_weak} is defined by
		\begin{align}
		(\text{P}_\varepsilon^k): \ a(u_\varepsilon^k,\varphi) + L\langle u_\varepsilon^k,\varphi\rangle_{L^2(\Gamma^{\varepsilon})} + M\langle  u_\varepsilon^k,\varphi\rangle_{L^2(\Omega^{\varepsilon})} = \langle f,\varphi \rangle_{L^2(\Omega^{\varepsilon})} + L\langle  u_\varepsilon^{k - 1},\varphi\rangle_{L^2(\Gamma^\varepsilon)} \\+ M\langle  u_\varepsilon^{k - 1}, \varphi\rangle_{L^2(\Omega^\varepsilon)} - \varepsilon^{\beta}\langle \mathcal{S}(u_{\varepsilon}^{k-1}),\varphi\rangle_{L^2(\Gamma^\varepsilon)} - \varepsilon^{\alpha}\langle \mathcal{R}(u_\varepsilon^{k-1}),\varphi\rangle_{L^2(\Omega^\varepsilon)}\nonumber
		\end{align}
		for all $\varphi \in \widetilde{\mathcal{W}}_\varepsilon$ and $k \in \mathbb{N}^{*}$ with the initial guess $u_\varepsilon^0 \in \mathcal{W}_\varepsilon$ chosen as $0$ and the stabilization constants $L,M > 0$ chosen later.
	\end{definition}
	\begin{lemma} \label{lem:1}
		Suppose $(\text{A}_1)$ and $(\text{A}_4)$ hold. Assume further that there exist constants $\underline{c}_L, \overline{c}_L, \underline{c}_M, \overline{c}_M>0$ independent of $\varepsilon$ such that
		\begin{align}
		\underline{c}_L(\varepsilon^\alpha + \varepsilon^\beta) \leq L \leq \overline{c}_L(\varepsilon^\alpha + \varepsilon^\beta),
		\; \underline{c}_M(\varepsilon^\alpha + \varepsilon^\beta) \leq M \leq \overline{c}_M(\varepsilon^\alpha + \varepsilon^\beta).
		\label{cond:1}
		\end{align}
		Denote by $(P_\varepsilon^1)$ the first-loop problem for $(P_\varepsilon^k)$ defined in Definition \ref{dn_app}.
		Then it admits a unique solution $u \in \widetilde{\mathcal{W}}_\varepsilon$ for each $\varepsilon >0$. 	
	\end{lemma}
	\begin{proof} Due to $(\text{A}_4)$ and the choice of $u_\varepsilon^0$, the problem $(P_\varepsilon^1)$ reads as
		\begin{align*}
		a(u_{\varepsilon}^{1},\varphi) & +L\langle u_{\varepsilon}^{1},\varphi\rangle_{L^{2}(\Gamma^{\varepsilon})}+M\langle u_{\varepsilon}^{1},\varphi\rangle_{L^{2}(\Omega^{\varepsilon})}\\
		& =\langle f,\varphi\rangle_{L^{2}(\Omega^{\varepsilon})}-\varepsilon^{\beta}\langle\mathcal{S}(0),\varphi\rangle_{L^{2}(\Gamma^{\varepsilon})}-\varepsilon^{\alpha}\langle\mathcal{R}(0),\varphi\rangle_{L^{2}(\Omega^{\varepsilon})}
		\end{align*}
		for all $\varphi \in \widetilde{\mathcal{W}}_\varepsilon$. Let us put $K_\varepsilon: \widetilde{\mathcal{W}}_\varepsilon \times \widetilde{\mathcal{W}}_\varepsilon \to \mathbb{R}$ given by
		\begin{align}
		K_\varepsilon(u,\varphi) := a(u,\varphi) + L\langle u,\varphi\rangle_{L^2(\Gamma^\varepsilon)} + M\langle u,\varphi\rangle_{L^2(\Omega^\varepsilon)} \quad \text{ for } u, \varphi \in \widetilde{\mathcal{W}}_\varepsilon.\nonumber
		\end{align}
		Clearly, this form is bilinear in $\widetilde{\mathcal{W}}_\varepsilon$ and its coerciveness is easily guaranteed. Therefore, by the Lax--Milgram argument there exists a unique $u \in \widetilde{\mathcal{W}}_\varepsilon$ satisfies $(P_\varepsilon^1)$.
	\end{proof}
	As a consequence of Lemma \ref{lem:1}, the sequence $\left\{u_\varepsilon^k\right\}_{k\in \mathbb{N}^*}$ is well-defined in $ \widetilde{\mathcal{W}}_{\varepsilon}$ under condition \eqref{cond:1}. The notion of having this assumption is transparent in the next theorem, where the choice of our stabilization terms is included.
	
	\begin{theorem}
		Assume $(\text{A}_1)$--$(\text{A}_4)$ hold. There exists a choice of $L$ and $M$ such that \eqref{cond:1} holds and the sequence of solution $\{u_\varepsilon^k\}_{k\in \mathbb{N}^*}$  to $(P_\varepsilon^k)$ defined in \eqref{dn_app} is Cauchy in $\widetilde{\mathcal{W}}_\varepsilon$. Moreover, the following estimate holds
		\begin{align}
		\|u_\varepsilon^{k + r +1} - u_\varepsilon^{k+1}\|_{{\widetilde{\mathcal{W}}}_\varepsilon}
		\le
		\frac{\eta^{k}(1 - \eta^r)C}{1 - \eta}\|u_\varepsilon^1\|_{\widetilde{\mathcal{W}}_\varepsilon},\nonumber
		\end{align}
		where $\eta \in (0,1)$ is  $\varepsilon$-independent and $k, r \in \mathbb{N}^*$.
	\end{theorem}
	\begin{proof}
		Define $v_\varepsilon^k := u_\varepsilon^k - u_\varepsilon^{k - 1} \in \widetilde{\mathcal{W}}_\varepsilon$ where $u_{\varepsilon}^{k}$ and $u_{\varepsilon}^{k-1}$ correspond to the solution of $(P_{\varepsilon}^{k})$ and $(P_{\varepsilon}^{k-1})$, respectively. Then we have the following difference equation 
		\begin{align}\label{14_1}
		a(v_\varepsilon^k,\varphi) + L\langle v_\varepsilon^k,\varphi\rangle_{L^2(\Gamma^\varepsilon)} + M\langle v_\varepsilon^k,\varphi\rangle_{L^2(\Omega^\varepsilon)} = -  \varepsilon^{\beta}\langle \mathcal{S}(u_\varepsilon^{k-1}) - \mathcal{S}(u_\varepsilon^{k-2}),\varphi\rangle_{L^2(\Gamma^\varepsilon)}  \\
		-\varepsilon^{\alpha}\langle \mathcal{R}(u_\varepsilon^{k-1}) - \mathcal{R}(u_\varepsilon^{k - 2}),\varphi\rangle_{L^2(\Omega^\varepsilon)} \;\text{ for all }\varphi \in \widetilde{\mathcal{W}}_\varepsilon.\nonumber
		\end{align}
		Choosing a test function $\varphi = v_\varepsilon^k \in \widetilde{\mathcal{W}}_\varepsilon$ in \eqref{14_1}, we see that
		\begin{align} \label{pt3}
		\int_{\Omega^\varepsilon}\mathbf{A}_\varepsilon(x)\left|\nabla v_\varepsilon^k\right|^2dx + L\|v_\varepsilon^k\|_{L^2(\Gamma^{\varepsilon})}^2 +  M\|v_\varepsilon^k\|_{L^2(\Omega^{\varepsilon})}^2 = \varepsilon^{\beta}\langle \mathcal{S}(u_\varepsilon^{k-2}) \\- \mathcal{S}(u_\varepsilon^{k-1}),v_\varepsilon^k\rangle_{L^2(\Gamma^\varepsilon)} + \varepsilon^{\alpha}\langle \mathcal{R}(u_\varepsilon^{k-2}) - \mathcal{R}(u_\varepsilon^{k - 1}),v_\varepsilon^k\rangle_{L^2(\Omega^\varepsilon)} .\nonumber
		\end{align}
		Now, we define 
		$h(t) := \varepsilon^{\beta}\mathcal{S}(t) - Lt \text{ and }  g(t) := \varepsilon^{\alpha}\mathcal{R}(t) - Mt$. Then taking the first-order derivative of $h$ and $g$ with respect to $t$, we get
		\begin{align}\label{15_1}
		h'(t) = \varepsilon^{\beta}\mathcal{S}'(t) - L \text{ and }
		g'(t) = \varepsilon^{\alpha}\mathcal{R}'(t) - M.
		\end{align}
		Notice that because of the structure of $h$ and $g$, \eqref{pt3} becomes
		\begin{align}\label{16_1}
		& \int_{\Omega^\varepsilon}\mathbf{A}_\varepsilon(x)\left|\nabla v_\varepsilon^k\right|^2dx + L\|v_\varepsilon^k\|_{L^2(\Gamma^{\varepsilon})}^2 +  M\|v_\varepsilon^k\|_{L^2(\Omega^{\varepsilon})}^2 \\&= -\langle h(u_\varepsilon^{k-1}) - h(u_\varepsilon^{k-2}),v_\varepsilon^k\rangle_{L^2(\Gamma^\varepsilon)}  - \langle g(u_\varepsilon^{k-1}) - g(u_\varepsilon^{k - 2}),v_\varepsilon^k\rangle_{L^2(\Omega^\varepsilon) }.\nonumber
		\end{align}
		At this stage, we have to choose $L$ and $M$ such that 
		\begin{equation}
		L \geq \delta_1\varepsilon^\beta \text{ and } M \geq \delta_1 \varepsilon^\alpha.
		\label{17_1}
		\end{equation}
		As a result, $h'$ and $g'$ computed in \eqref{15_1} can be bounded with the help of $(\mathbf{A}_3)$ by
		\begin{align*}
		\varepsilon^\beta\delta_0 - L \leq h'(t) \leq 0,
		\varepsilon^\beta\delta_0 - M \leq g'(t) \leq 0 \quad \text{ a.e in } \mathbb{R},
		\end{align*}
		or it is equivalent to
		\begin{align}\label{18_1}
		\left|h'(t)\right| \leq  L - \varepsilon^\beta\delta_0 \text{ and } \left|g'(t)\right| \leq M - \varepsilon^\beta \delta_0 \quad \text{ a.e in } \mathbb{R}.
		\end{align}
		
		Combining \eqref{18_1} and $(\mathbf{A}_1)$, \eqref{16_1} leads to the following estimate:
		\begin{align*}
		& \underline{\gamma}\|\nabla v_\varepsilon^k\|_{L^2(\Omega^{\varepsilon})}^2 + L\|v_{\varepsilon}^k\|_{L^2(\Gamma^{\varepsilon})}^2 + M\|v_{\varepsilon}^k\|_{L^2(\Omega^{\varepsilon})}^2 \\ & \leq (L - \delta_0\varepsilon^{\beta})\|v_{\varepsilon}^{k-1}\|_{L^2(\Gamma^{\varepsilon})}\|v_{\varepsilon}^k\|_{L^2(\Gamma^\varepsilon)} + (M - \delta_0\varepsilon^{\beta})\|v_{\varepsilon}^{k-1}\|_{L^2(\Omega^{\varepsilon})}\|v_{\varepsilon}^k\|_{L^2(\Omega^\varepsilon)}.
		\end{align*}
		Using the Cauchy--Schwarz inequality, we have that
		\begin{align}\label{ineq_bfomit}
		&
		\underline{\gamma}\|\nabla v_\varepsilon^k\|_{L^2(\Omega^{\varepsilon})}^2 + \frac{L+\delta_0\varepsilon^\beta}{2}\|v_{\varepsilon}^k\|_{L^2(\Gamma^{\varepsilon})}^2 + \frac{M + \delta_0\varepsilon^\alpha}{2}\|v_{\varepsilon}^k\|_{L^2(\Omega^{\varepsilon})}^2 \\ &\leq \frac{L - \delta_0\varepsilon^{\beta}}{2}\|v_\varepsilon^{k - 1}\|_{L^2(\Gamma^\varepsilon)}^2 + \frac{M - \delta_0\varepsilon^{\alpha}}{2}\|v_{\varepsilon}^{k-1}\|_{L^2(\Omega^{\varepsilon})}^2 . \nonumber
		\end{align}
		Omitting the first term of the left-hand side of \eqref{ineq_bfomit}, we obtain
		\begin{align}\label{24_1}
		\|v_{\varepsilon}^k\|_{L^2(\Gamma^{\varepsilon})}^2 + \|v_{\varepsilon}^k\|_{L^2(\Omega^{\varepsilon})}^2 \leq \frac{L - \delta_0\varepsilon^{\beta}}{\min\{L + \delta_0\varepsilon^{\beta},M + \delta_0\varepsilon^{\alpha}\}}\|v_\varepsilon^{k - 1}\|_{L^2(\Gamma^{\epsilon})}^2 \\+ \frac{M - \delta_0\varepsilon^{\alpha}}{\min\{L + \delta_0\varepsilon^{\beta},M + \delta_0\varepsilon^{\alpha}\}}\|v_{\varepsilon}^{k-1}\|_{L^2(\Omega^{\varepsilon})}^2. \nonumber
		\end{align}
		
		Rewriting \eqref{24_1}, we thus have
		\begin{align}\label{ine_max}
		\|v_{\varepsilon}^k\|_{L^2(\Gamma^{\varepsilon})}^2 + \|v_{\varepsilon}^k\|_{L^2(\Omega^{\varepsilon})}^2 \leq \eta^2_\varepsilon(\alpha,\beta)\left(\|v_\varepsilon^{k - 1}\|_{L^2(\Gamma^{\epsilon})}^2 + \|v_{\varepsilon}^{k-1}\|_{L^2(\Omega^{\varepsilon})}^2\right),
		\end{align}
		where we have denoted by $$\eta_\varepsilon(\alpha,\beta) := \left(\max\left\{\frac{L - \delta_0\varepsilon^{\beta}}{L + \delta_0\varepsilon^{\beta}},\frac{L - \delta_0\varepsilon^{\beta}}{M + \delta_0\varepsilon^\alpha},\frac{M - \delta_0\varepsilon^{\alpha}}{M + \delta_0\varepsilon^{\alpha}},\frac{M - \delta_0\varepsilon^{\alpha}}{L + \delta_0\varepsilon^\beta}\right\}\right)^{1/2} .$$
		
		According to the linearization procedures, we need to find an $\varepsilon$-independent bound for $\eta_\varepsilon$ in \eqref{ine_max} such that it is strictly less than $1$.
		Accordingly, we choose the stabilization constants $L$ and $M$ such that $\eta_\varepsilon < 1$ for all $\varepsilon > 0$ and $\alpha, \beta \in \mathbb{R}$. Now, we may write $\eta_\varepsilon = \eta$, i.e. it is independent of $\varepsilon$ by suitable choices of $L, M$. Note that 
		\begin{align}\label{star}
		\frac{L - \delta_0\varepsilon^\beta}{L + \delta_0\varepsilon^\beta} < 1 \text{ and } \frac{M - \delta_0\varepsilon^\alpha}{M + \delta_0\varepsilon^\alpha} < 1
		\end{align}
		because of the choice \eqref{18_1}. Therefore, we target the following cases:  
		\begin{align}
		\frac{L - \delta_0\varepsilon^\beta}{M + \delta_0\varepsilon^\alpha} < 1 \ \text{ and } \ \frac{M - \delta_0\varepsilon^\alpha}{L + \delta_0\varepsilon^\alpha} < 1.\nonumber
		\end{align}
		
		From \eqref{star}, a suitable choice of $L,M$ is taking 
		\begin{align}
		M + \delta_0\varepsilon^{\alpha} \geq L + \delta_0\varepsilon^{\beta}
		\text{ and }
		L + \delta_0\varepsilon^{\beta} \geq M + \delta_0\varepsilon^{\alpha},\nonumber
		\end{align}
		which leads to $M - L = \delta_0(\varepsilon^{\beta} - \varepsilon^{\alpha})$. Hence, in accordance with \eqref{17_1} the suitable choice we eventually obtain is provided as follows:
		\begin{equation}
		L=L\left(\varepsilon\right):=\begin{cases}
		\delta_{1}\varepsilon^{\alpha}+\delta_{0}\left(\varepsilon^{\alpha}-\varepsilon^{\beta}\right) & \text{if }\alpha\le\beta,\\
		\delta_{1}\varepsilon^{\beta} & \text{if }\alpha\ge\beta,
		\end{cases}
		\label{choice1}
		\end{equation}
		\begin{equation}
		M=M\left(\varepsilon\right):=\begin{cases}
		\delta_{1}\varepsilon^{\alpha} & \text{if }\alpha\le\beta,\\
		\delta_{1}\varepsilon^{\beta}+\delta_{0}\left(\varepsilon^{\beta}-\varepsilon^{\alpha}\right) & \text{if }\alpha\ge\beta.
		\end{cases}
		\label{choice2}
		\end{equation}
		
		Interestingly, this choice works for all real scaling parameters $\alpha,\beta$. It also agrees with \eqref{17_1} and guarantees the positivity of such stabilization constants. In addition, we now observe that \eqref{choice1} and \eqref{choice2} are well-suited to the condition \eqref{cond:1} in Lemma \ref{lem:1}, where the well-posedness of the first-loop problem of $(P_\varepsilon^k)$ is proven. Collectively, we have demonstrated that there exists a choice of $L$ and $M$ satisfying \eqref{17_1} such that $\eta_\varepsilon = \eta < 1$ for all scaling factor $\varepsilon > 0$ and scaling parameters $\alpha,\beta\in\mathbb{R}$.
		
		As a consequence of \eqref{ine_max} and \eqref{choice1}--\eqref{choice2}, we conclude that for every $\varepsilon>0$ and $k\in\mathbb{N}^{*}$, the following estimate holds
		\begin{align}\label{ine_max-1}
		(\varepsilon^\alpha + \varepsilon^\beta) \|v_{\varepsilon}^k\|_{\mathcal{W}_{\varepsilon}}^2
		\leq \eta^2 (\varepsilon^\alpha + \varepsilon^\beta) \|v_\varepsilon^{k - 1}\|_{\mathcal{W}_{\varepsilon}}^2.
		\end{align}
		On the other hand, for any $k,r\in\mathbb{N}^*$ we have
		\begin{align}	\label{ine_max-2}
		&\sqrt{\varepsilon^\alpha + \varepsilon^\beta}\|u_\varepsilon^{k + r} - u_\varepsilon^k\|_{\mathcal{W}_{\varepsilon}} 
		\\
		&\leq \sqrt{\varepsilon^\alpha + \varepsilon^\beta}\|u_\varepsilon^{k + r} - u_\varepsilon^{k+ r-1}\|_{\mathcal{W}_{\varepsilon}}+ \ldots + \sqrt{\varepsilon^\alpha + \varepsilon^\beta}\|u_\varepsilon^{k+1} - u_\varepsilon^k\|_{\mathcal{W}_{\varepsilon}}\nonumber\\
		&\leq \sqrt{\varepsilon^\alpha + \varepsilon^\beta}(\eta^{k + r - 1} + \ldots + \eta^{k})\|u_\varepsilon^1 - u_\varepsilon^0\|_{\mathcal{W}_\varepsilon}\leq \frac{\eta^{k}(1 - \eta^r)}{1 - \eta}\sqrt{\varepsilon^\alpha + \varepsilon^\beta}\|u_\varepsilon^1\|_{\mathcal{W}_\varepsilon}.\nonumber
		\end{align}
		
		From now on, it remains to estimate the difference gradient we omitted in \eqref{ineq_bfomit}. Once again, it follows from \eqref{ineq_bfomit} and \eqref{choice1}--\eqref{choice2} that
		\begin{align}
		\underline{\gamma}\|\nabla v_\varepsilon^k\|_{L^2(\Omega^\varepsilon)}^2 
		&\leq \frac{L - \delta_0\varepsilon^\beta}{2}\|v_\varepsilon^{k-1}\|_{L^2(\Gamma^\varepsilon)}^2 + \frac{M - \delta_0\varepsilon^\alpha}{2}\|v_\varepsilon^{k - 1}\|_{L^2(\Omega^\varepsilon)}^2 \nonumber
		\\&
		\leq \frac{\delta_1+\delta_0}{2}(\varepsilon^\alpha + \varepsilon^\beta) \|v_\varepsilon^{k - 1}\|_{\mathcal{W}_{\varepsilon}}^2.\nonumber
		\end{align}
		We then combine this with \eqref{ine_max-1} to get
		\begin{align}
		\|\nabla v_\varepsilon^{k+1}\|_{L^2(\Omega^\varepsilon)}^2 
		\leq \eta^2 \frac{\delta_1+\delta_0}{2\underline{\gamma}}(\varepsilon^\alpha + \varepsilon^\beta) \|v_\varepsilon^{k }\|_{\mathcal{W}_{\varepsilon}}^2,\nonumber
		\end{align}
		and as a by-product, it yields
		\begin{equation}
		\|v_\varepsilon^{k + 1}\|_{{\widetilde{\mathcal{W}}}_\varepsilon}^2
		\le 
		\eta^2 \left(\frac{\delta_1+\delta_0}{2\underline{\gamma}} + 1\right)
		\|v_\varepsilon^{k}\|_{{\widetilde{\mathcal{W}}}_\varepsilon}^2.\nonumber
		\end{equation}
		At this moment, we proceed as in \eqref{ine_max-2} to arrive at
		\begin{equation}
		\|u_\varepsilon^{k + r + 1} - u_\varepsilon^{k+1}\|_{{\widetilde{\mathcal{W}}}_\varepsilon}
		\le
		\frac{\eta^{k}(1 - \eta^r)}{1 - \eta}\left(\frac{\delta_1+\delta_0}{2\underline{\gamma}} + 1\right)^{1/2}
		\|u_\varepsilon^1\|_{{\widetilde{\mathcal{W}}}_\varepsilon}.
		\label{finale}
		\end{equation}
		
		This completes the proof of the theorem.
	\end{proof}
	It is worth noting that from \eqref{ine_max-2}, the iterative sequence $\{u_\varepsilon^k\}_{k\in \mathbb{N}^*}$ is Cauchy in $\mathcal{W}_{\varepsilon}$ for any $\varepsilon>0$ when $\alpha = \beta = 0$.
	From \eqref{finale}, this sequence is Cauchy in $\widetilde{\mathcal{W}}_{\varepsilon}$. Consequently, there exists a unique $u_{\varepsilon}\in \widetilde{\mathcal{W}}_{\varepsilon}$ such that $u_{\varepsilon}^{k}\to u_{\varepsilon}$ as $k\to \infty$ and for each $\varepsilon>0$. On the other side, using the Lipschitz properties of the volume and surface reaction rates assumed in ($\text{A}_2$), we have
	\begin{align}
	& \varepsilon^{\beta}\mathcal{S}\left(u_{\varepsilon}^{k}\right)\to\varepsilon^{\beta}\mathcal{S}\left(u_{\varepsilon}\right)\;\text{strongly in }L^{2}\left(\Gamma^{\varepsilon}\right), \label{strong1}\\
	& \varepsilon^{\alpha}\mathcal{R}\left(u_{\varepsilon}^{k}\right)\to\varepsilon^{\alpha}\mathcal{R}\left(u_{\varepsilon}\right)\;\text{strongly in }L^{2}\left(\Omega^{\varepsilon}\right) \text{ as } k\to \infty.
	\label{strong2}
	\end{align}
	
	Hence, $u_\varepsilon$ is a unique solution of the microscopic model $(P_\varepsilon)$ in the sense of Definition \ref{dn_weak}. Besides, when taking $r\to \infty$ in \eqref{finale}, its stability is confirmed by
	\begin{align}
	\|u_\varepsilon^k - u_\varepsilon \|_{{\widetilde{\mathcal{W}}}_\varepsilon}^2 \leq C\eta^{2(k - 1)}\|u_\varepsilon^1\|_{\widetilde{\mathcal{W}}_\varepsilon}^2.\nonumber
	\end{align}
	
	As a result, we state the following theorem.
	\begin{theorem}\label{mainthm-2}
		Assume $(A_1)$--$(A_4)$ hold. Then for each $\varepsilon >0$ there exists a unique solution of $(P_\varepsilon)$ in the sense of Definition \ref{dn_weak}.
	\end{theorem}
	\begin{remark}
		Compared to the mild restriction in \cite{KM16-1,KM16}, we no longer rely on the Poincar\'e constant and the lower bound of the diffusion. However, information of reaction terms in this context is further required as specified in ($\text{A}_3$). 

	\end{remark}
	\section{ Asymptotic behaviors and convergence results}\label{sec:4}
	\subsection{Volume reaction and surface reaction}\label{sec:4.1}
	
	In this subsection, we aim to see the asymptotic behaviors of the microscopic solution of $(P_{\varepsilon})$, when the volume and surface reactions are involved separately. In other words, as the starting point we consider the following problems:
	\[
	\left(P_{\varepsilon}^{R}\right):\;\begin{cases}
	\nabla\cdot\left(-\mathbf{A}\left(x/\varepsilon\right)\nabla u_{\varepsilon}\right) + \varepsilon^{\alpha}\mathcal{R}\left(u_{\varepsilon}\right)= f & \text{in}\;\Omega^{\varepsilon},\\
	-\mathbf{A}\left(x/\varepsilon\right)\nabla u_{\varepsilon}\cdot\text{n}=0 & \text{across}\;\Gamma^{\varepsilon},\\
	u_{\varepsilon}=0 & \text{across}\;\Gamma^{\text{ext}},
	\end{cases}
	\]
	\[
	\left(P_{\varepsilon}^{S}\right):\;\begin{cases}
	\nabla\cdot\left(-\mathbf{A}\left(x/\varepsilon\right)\nabla u_{\varepsilon}\right)=f & \text{in}\;\Omega^{\varepsilon},\\
	-\mathbf{A}\left(x/\varepsilon\right)\nabla u_{\varepsilon}\cdot\text{n}=\varepsilon^{\beta}\mathcal{S}\left(u_{\varepsilon}\right) & \text{across}\;\Gamma^{\varepsilon},\\
	u_{\varepsilon}=0 & \text{across}\;\Gamma^{\text{ext}}.
	\end{cases}
	\]
	
	
	\subsubsection{Volume reaction}
	
	\subsubsection*{The case $\alpha > 0$}
	
	Given a natural constant $\theta\ge2$, we define
	the index set
	\begin{equation}\label{Mo}
	\mathcal{M}_{\alpha,\theta}:=\left\{ k,l\in\left[0,\theta\right]:k\alpha+l\ge1\;\text{and}\;k+l \leq \theta\right\}.
	\end{equation}
	
	The asymptotic expansion we consider here is structured as follows: 
	\begin{align}\label{expansion_1}
	u_{\varepsilon}\left(x\right)=u_{0}\left(x,y\right)+\varepsilon^{\alpha}u_{1,-1}\left(x,y\right)+\sum_{\left(k,l\right)\in\mathcal{M}_{\alpha,\theta}}\varepsilon^{k\alpha+l}u_{k,l}\left(x,y\right)+\mathcal{O}\left(\varepsilon^{\theta+1}\right),
	\end{align}
	where $x\in \Omega^{\varepsilon},y\in Y_{l}$ and all components $u_{k,l}$ are periodic in $y$.
	
	\begin{remark}
		This ansatz essentially mimics the standard two-scale asymptotic expansions used in the homogenization theory for second-order elliptic equations. Since the diffusion coefficient of the PDE is periodic in $y$, it is reasonable to require that all $u_{k,l}$ are periodic functions of $y$. For $\varepsilon \ll 1$ the microscopic variable $y$ changes much more rapidly than $x$ and heuristically, the macroscopic variable can be viewed as a ``constant'', when looking at the microscopic problem. This is why the method is expected to treat the ``slow'' variable $x$ and the ``fast'' one $y$ independently. Furthermore, this way the gradient operator and the gradient of the fluxes can be evaluated according to the rule $\nabla = \nabla_{x} + \varepsilon^{-1}\nabla_{y}$. With this essence in mind, our designation of asymptotic expansions in this work are such that we are able to handle variable scalings $\alpha$ and $\beta$ in the PDE in the rigorous asymptotic analysis. From the physical point of view, the use of asymptotic expansions in understanding size effects in periodic media was studied in e.g. \cite{Triantafyllidis1996}.
	\end{remark}
	
	Assume that there exists a Lipschitz-continuous function $\bar{\mathcal{R}}$ such that
	\begin{equation}
	\mathcal{R}(u^\varepsilon) = \bar{\mathcal{R}}(u_0) + \varepsilon^{\alpha} \bar{\mathcal{R}}(u_{1,-1}) + 
	\sum_{(k,l) \in \mathcal{M}_{\alpha,\theta}} \varepsilon^{k(\alpha + 1)+l } \bar{\mathcal{R}}(u_{k,l}) +
	\mathcal{O}(\varepsilon^{\theta + 1}).
	\label{eq:conditionR}
	\end{equation}
	This corresponds to the fact that there exists $\mathbf{L}_{\mathcal{R}}>0$ such that
	\begin{equation}
	\left\Vert \bar{\mathcal{R}}(u) - \bar{\mathcal{R}}(v)\right\Vert_{L^2(\Omega^{\varepsilon})}
	\le \mathbf{L}_{\mathcal{R}} \left\Vert u - v\right\Vert_{L^2(\Omega^{\varepsilon})} \text{ for }   u,v\in\mathbb{R}.
	\label{eq:conditionR-1}
	\end{equation}
	
	\begin{remark}
		In the sequel, our new assumptions \eqref{eq:conditionR} and \eqref{eq:conditionR-1} on the  reaction rate $\mathcal{R}$ are termed as ($\text{A}_5$) and ($\text{A}_6$), respectively. It resembles the definition of almost additive functions with positive homogeneity in stochastic processes (see, e.g., \cite{SSV17}).
	\end{remark}
	
	Due to the simple relation $\nabla = \nabla_{x} + \varepsilon^{-1}\nabla_{y}$, it follows that
	\begin{align*}
	\nabla u_{\varepsilon} & =\varepsilon^{-1}\nabla_{y}u_{0}+\varepsilon^{\alpha-1}\nabla_{y}u_{1,-1}+\varepsilon^{0}\left(\nabla_{x}u_{0}+\nabla_{y}u_{0,1}\right)\\
	& +\varepsilon^{\alpha}\left(\nabla_{x}u_{1,-1}+\nabla_{y}u_{1,0}\right)+\sum_{\left(k,l\right)\in\mathcal{N}_{\alpha,\theta}}\varepsilon^{k\alpha+l}\left(\nabla_{x}u_{k,l}+\nabla_{y}u_{k,l+1}\right)+\mathcal{O}\left(\varepsilon^{\theta}\right),
	\end{align*}
	where $\mathcal{N}_{\alpha,\theta}=\mathcal{M}_{\alpha,\theta}\backslash\left\{ \left(0,1\right)\right\} $. Hereafter, the diffusion term involved in $\left(P_{\varepsilon}^{R}\right)$ is expressed as
	\begin{align}\label{eq:diffusion}
	&\nabla\cdot\left(-\mathbf{A}\left(y\right)\nabla u_{\varepsilon}\right) \\ &  =\varepsilon^{-2}\nabla_{y}\cdot\left(-\mathbf{A}\left(y\right)\nabla_{y}u_{0}\right)+\varepsilon^{\alpha-2}\nabla_{y}\cdot\left(-\mathbf{A}\left(y\right)\nabla_{y}u_{1,-1}\right)
	\nonumber\\
	& +\varepsilon^{-1}\left(\nabla_{x}\cdot\left(-\mathbf{A}\left(y\right)\nabla_{y}u_{0}\right)+\nabla_{y}\cdot\left(-\mathbf{A}\left(y\right)\left(\nabla_{x}u_{0}+\nabla_{y}u_{0,1}\right)\right)\right)
	\nonumber\\
	& +\varepsilon^{\alpha-1}\left(\nabla_{x}\cdot\left(-\mathbf{A}\left(y\right)\nabla_{y}u_{1,-1}\right)+\nabla_{y}\cdot\left(-\mathbf{A}\left(y\right)\left(\nabla_{x}u_{1,-1}+\nabla_{y}u_{1,0}\right)\right)\right)
	\nonumber\\
	& +\varepsilon^{0}\left(\nabla_{x}\cdot\left(-\mathbf{A}\left(y\right)\left(\nabla_{x}u_{0}+\nabla_{y}u_{0,1}\right)\right)+\nabla_{y}\cdot\left(-\mathbf{A}\left(y\right)\left(\nabla_{x}u_{0,1}+\nabla_{y}u_{0,2}\right)\right)\right)
	\nonumber\\
	& +\varepsilon^{\alpha}\left(\nabla_{x}\cdot\left(-\mathbf{A}\left(y\right)\left(\nabla_{x}u_{1,-1}+\nabla_{y}u_{1,0}\right)\right)+\nabla_{y}\cdot\left(-\mathbf{A}\left(y\right)\left(\nabla_{x}u_{1,0}+\nabla_{y}u_{1,1}\right)\right)\right)
	\nonumber\\
	& +\sum_{\left(k,l\right)\in\mathcal{N}_{\alpha,\theta}}\varepsilon^{k\left(\alpha+1\right)+l}\left[\nabla_{x}\cdot\left(-\mathbf{A}\left(y\right)\left(\nabla_{x}u_{k,l}+\nabla_{y}u_{k,l+1}\right)\right)\right.
	\nonumber\\
	& \left.+\nabla_{y}\cdot\left(-\mathbf{A}\left(y\right)\left(\nabla_{x}u_{k,l+1}+\nabla_{y}u_{k,l+2}\right)\right)\right]+\mathcal{O}\left(\varepsilon^{\theta-1}\right),\nonumber
	\end{align}
	while relying on ($\text{A}_5$), the reaction term can be decomposed as
	\begin{align}\label{Ru}
	\varepsilon^\alpha\mathcal{R}(u^{\varepsilon}) = 
	\varepsilon^\alpha\bar{\mathcal{R}}(u_0) + \varepsilon^{2\alpha}\bar{\mathcal{R}}(u_{1,-1}) + \sum_{(k,l) \in \mathcal{M}_{\alpha,\theta}}\varepsilon^{k(\alpha + 1) + l + \alpha}\bar{\mathcal{R}}(u_{k,l}) + \mathcal{O}(\varepsilon^{\theta+1}).
	\end{align}
	
	In the same vein, the term on internal micro-surfaces are determined by
	\begin{align}\label{eq:boundary}
	-\mathbf{A}(y)\nabla u_\varepsilon \cdot \text{n} &= -\varepsilon^{-1}\mathbf{A}(y)u_0\cdot \text{n} -\varepsilon^{\alpha - 1}\mathbf{A}\nabla_{y}u_{1,-1}\cdot \text{n} \\&- \varepsilon^0\mathbf{A}(y)(\nabla_{x}u_0 + \nabla_{y}u_{0,1})\cdot \text{n} - \varepsilon^{\alpha}\mathbf{A}(y)(\nabla_{x}u_{1,-1} + \nabla_{y}u_{1,0})\cdot \text{n} \nonumber\\&-\varepsilon\mathbf{A}(y)(\nabla_x u_{0,1}+ \nabla_y u_{0,2})\cdot \text{n} \nonumber
	- \varepsilon^{\alpha + 1}\mathbf{A}(y)(\nabla_x u_{1,0} + \nabla_y u_{1,1})\cdot \text{n} \nonumber\\&-\sum_{(k,l)\in \mathcal{K}_{\alpha,\theta}}\varepsilon^{k(\alpha + 1) + l}\mathbf{A}(y)\left(\nabla_{x}u_{k,l} + \nabla_{y}u_{k,l+1}\right)\cdot\text{n} + \mathcal{O}(\varepsilon^\theta),\nonumber
	\end{align}
	where $\mathcal{K}_{\alpha,\theta} = \mathcal{N}_{\alpha,\theta} \backslash \{(1,0)\}$. From now on, we set:
	\begin{align}\label{operators}
	\mathcal{A}_0 &:= \nabla_{y}\cdot(-\mathbf{A}(y)\nabla_{y}),  \\
	\mathcal{A}_1 &:= \nabla_{x}\cdot(-\mathbf{A}(y)\nabla_{y}) + \nabla_{y}\cdot(-\mathbf{A}(y)\nabla_{x}),\nonumber\\
	\mathcal{A}_2&:= \nabla_{x}\cdot(-\mathbf{A}(y)\nabla_{x}).\nonumber
	\end{align}
	
	We obtain the following auxiliary problems from  \eqref{eq:diffusion} and \eqref{eq:boundary}:
	\begin{align}
	(\varepsilon^{-2}): \quad &\begin{cases}
	\mathcal{A}_0u_0 = 0 \quad \text{ in } Y_l,\\
	-\mathbf{A}(y)\nabla_{y}u_0
	\cdot \text{n} =0 \quad \text{ on } \Gamma,\\
	u_0 \text{ is periodic in } y,
	\end{cases}\label{auxi_0}\\
	(\varepsilon^{\alpha - 2}): \quad &\begin{cases}
	\mathcal{A}_0 u_{1,-1} = 0 \quad \text{ in } Y_l,\\
	-\mathbf{A}(y)\nabla_y u_{1,-1}\cdot \text{n} = 0 \quad \text{ on } \Gamma,\\
	u_{1,-1} \text{ is periodic in } y,
	\end{cases}\label{auxi_1}\\
	(\varepsilon^{-1}): \quad &\begin{cases}
	\mathcal{A}_0 u_{0,1} = -\mathcal{A}_1 u_0 \quad \text{ in } Y_l,\\
	-\mathbf{A}(y)(\nabla_x u_0 + \nabla_y u_{0,1})\cdot \text{n} = 0 \quad \text{ on } \Gamma,\\
	u_{0,1} \text{ is periodic in } y,
	\end{cases}\label{auxi_2}\\
	(\varepsilon^{\alpha - 1}): \quad &\begin{cases}
	\mathcal{A}_0 u_{1,0} = -\mathcal{A}_1 u_{1,-1} \quad \text{ in } Y_l,\\
	-\mathbf{A}(y)(\nabla_x u_{1,-1} + \nabla_y u_{1,0}) \cdot \text{n} = 0 \quad \text{ on } \Gamma,\\
	u_{1,0} \text{ is periodic in } y,
	\end{cases}\label{auxi_3}\\
	(\varepsilon^0): \quad &\begin{cases}
	\mathcal{A}_0 u_{0,2}  = f - \mathcal{A}_1 u_{0,1} - \mathcal{A}_2u_0 \quad \text{ in } Y_l,\\
	-\mathbf{A}(y)(\nabla_x u_{0,1} + \nabla_y u_{0,2}) \cdot \text{n} = 0 \quad \text{ on } \Gamma,\\
	u_{0,2} \text{ is periodic in } y,
	\end{cases}\label{auxi_4}\\
	(\varepsilon^{\alpha}): \quad&\begin{cases}
	\mathcal{A}_0u_{1,1} = - \mathcal{A}_1u_{1,0} - \mathcal{A}_2u_{1,-1} \quad \text{ in } Y_l,\\
	-\mathbf{A}(y)(\nabla_{x}u_{1,0} +\nabla_{y}u_{1,1}) \cdot \text{n} = 0 \quad \text{ on } \Gamma,\\
	u_{1,1} \text{ is periodic in } y.
	\end{cases}\label{auxi_5}\\
	&\vdots \nonumber\\
	(\varepsilon^{k(\alpha+1)+l}): \quad&\begin{cases}
	\mathcal{A}_0u_{k,l+2} = - \mathcal{A}_1u_{k,l+1} - \mathcal{A}_2u_{k,l} \quad \text{ in } Y_l,\\
	-\mathbf{A}(y)(\nabla_{x}u_{k,l+1} + \nabla_{y}u_{k,l+2})\cdot \text{n} = 0 \quad \text{ on } \Gamma,\\
	u_{k,l+2} \text{ is periodic in } y,
	\end{cases}\label{auxi_6}
	\end{align}
	for all pairs $(k,l) \in \mathcal{K}_{\alpha,\theta-2}$.
	
	By classical arguments in homogenization procedures, one has  from \eqref{auxi_0} and \eqref{auxi_1} that $u_0$ and $u_{1,-1}$ are independent of $y$. Without loss of generality, we  take $u_{1,-1}\equiv 0$ and by substitution, we also get $u_{1,0}\equiv 0$ in \eqref{auxi_3}. Besides, we write
	\begin{align}\label{u0}
	u_{0}(x,y) = \tilde{u}_0(x).
	\end{align}
	Therefore, the auxiliary problem \eqref{auxi_2} is solvable with respect to $u_{0,1}$. Plugging all auxiliary solutions that have been deduced above into \eqref{auxi_4} and \eqref{auxi_5}, we easily obtain $u_{0,2}$ and $u_{1,1}$. On the whole, we repeat the same strategy and ensure the solvability of the high-order auxiliary problem \eqref{auxi_6}. From e.g. \cite{CP99}, the existence and uniqueness results for \eqref{auxi_2} are trivial and the solution $u_{0,1}$ is sought in the sense of separation of variables. In other words, we have that
	\begin{align}\label{u_01}
	u_{0,1}(x,y) = -\chi_{0,1}(y)\cdot\nabla_{x}\tilde{u}_{0}(x).
	\end{align}
	
	Hereby, the following cell problem for the field $\chi_{0,1}(y)$ is obtained:
	\begin{align}\label{chi01}
	\begin{cases}
	\mathcal{A}_{0}\chi_{0,1}^{j}=\dfrac{\partial\mathbf{A}_{ij}}{\partial y_{i}} & \text{in }Y_{l},\\
	-\mathbf{A}\left(y\right)\nabla_{y}\chi_{0,1}^{j}\cdot\text{n}=\mathbf{A}\left(y\right)\cdot\text{n}_j & \text{on }\Gamma,\\
	\chi_{0,1}^{j}\text{ is periodic},
	\end{cases}
	\end{align}
	where $\mathbf{A}_{ij}$ are elements of the second-order tensor $\mathbf{A}$ with $1\le i,j\le d$ and $\chi_{0,1}^{j}=\chi_{0,1}^{j}(y)$ are elements in the cell vector-valued function $\chi_{0,1}$. Remarkably, classical results provide that $\chi_{0,1}\in [H^{1}_{\#}(Y_{l})/ \mathbb{R}]^{d}$ exists uniquely in these cell problems.

	From the cell function $\chi_{0,1}$ in \eqref{u_01}, we obtain the limit equation by taking into account the auxiliary problem \eqref{auxi_4}. In fact, the limit equation is of the following structure:
	\begin{align}\label{limit_construct1-1}
	-\bar{\mathbf{A}}\Delta_{x}\tilde{u}_{0}=\frac{\left|Y_{l}\right|}{\left|Y\right|}f\quad\text{in }\Omega,
	\end{align}
	where the coefficient $\left|Y_{l}\right|/\left|Y\right|$ is referred to as the volumetric porosity and $\bar{\mathbf{A}}$ given by
	\begin{align}\label{limit_construct2-1}
	\bar{\mathbf{A}}=\frac{1}{\left|Y\right|}\int_{Y_{l}}\left(-\nabla_{y}\chi_{0,1}\left(y\right)+\mathbb{I}\right)\mathbf{A}\left(y\right)dy,
	\end{align}
	is the effective diffusion coefficient corresponding to $\mathbf{A}$ with $\mathbb{I}$ being the identity matrix.
	
	Obviously, this limit equation is supplemented with the zero Dirichlet boundary condition on $\Gamma^{\text{ext}}$ and $\bar{\mathbf{A}}$ satisfies the ellipticity condition (cf. \cite[Proposition 2.6]{CP99}).
	
	\begin{remark}\label{remarkCinf} We recall from \cite{Khoa17} that when $\alpha = 0$, the limit equation becomes semi-linear, i.e.
		\begin{align}\label{limit_construct1}
		-\bar{\mathbf{A}}\Delta_{x}\tilde{u}_{0}-\frac{\left|Y_{l}\right|}{\left|Y\right|}\bar{\mathcal{R}}\left(\tilde{u}_{0}\right) = 0\quad\text{in }\Omega,
		\end{align}
		where we have omitted $f$, for simplicity.
		Based on the Lax--Milgram argument, the limit problem \eqref{limit_construct1} with the homogeneous Dirichlet boundary condition admits a unique solution $\tilde{u}_0 \in H_0^1(\Omega)$ due to the Lipschitz reaction term. Moreover, from \cite[Lemma 5]{GMRL09}, it is essentially bounded and the following estimate holds
		\begin{align}
		\|\tilde{u}_0\|_{L^\infty(\Omega)} \leq C(\|\tilde{u}_0\|_{L^2(\Omega)} + 1).\nonumber
		\end{align}
		Accordingly, these results can be applied to the limit problem \eqref{limit_construct1-1}, including the existence and uniqueness of $\tilde{u}_0 \in H_0^1(\Omega)$ for any $f \in L^2(\Omega)$. From \cite[Corollary 8.11]{GT83}, if $f\in C^\infty(\Omega)$, this solution $\tilde{u}_0$ belongs to $C^\infty(\Omega)$.
		
	\end{remark}
	
	Due to the structure of the auxiliary problems \eqref{auxi_0}--\eqref{auxi_6}, we get $u_{k,l}\equiv 0$ for all $k \ge 1$ and $(k,l) \in \mathcal{K}_{\alpha,\theta}$. In line with \cite{Khoa17}, we obtain when $k=0$ that
	\begin{align}\label{u_0l}
	u_{0,l} = (-1)^{l} \chi_{0,l}(y)\cdot\nabla_{x}^{l}\tilde{u}_{0}(x).
	\end{align}
	
	Thus, we obtain the following high-order cell problems in this case:
	\begin{align}\label{chigeneral}
	\begin{cases}
	\nabla_{y}\cdot\left(-\mathbf{A}\left(y\right)\left(\nabla_{y}\chi_{0,l+2}-\chi_{0,l+1}\right)\right)\nabla_{x}^{l+2}\tilde{u}_{0}\\
	=\left(-1\right)^{l}\bar{\mathcal{R}}\left(\left(-1\right)^{l}\chi_{0,l}\nabla_{x}^{l}\tilde{u}_{0}\right)-\left(\mathbf{A}\left(y\right)-\mathbb{I}\right)\nabla_{y}\chi_{0,l+1}\nabla_{x}^{l+2}\tilde{u}_{0} & \text{in }Y_{l},\\
	-\mathbf{A}\left(y\right)\left(\nabla_{y}\chi_{0,l+2}-\chi_{0,l+1}\right)\nabla_{x}^{l+2}\tilde{u}_{0}\cdot\text{n}=0 & \text{on }\Gamma,\\
	\chi_{0,l+2}\text{ is periodic}.
	\end{cases}
	\end{align}
	
	Combining \eqref{u0}, \eqref{u_01} and \eqref{u_0l} we recover the structure of the asymptotic expansion for $u^{\varepsilon}$ defined in \eqref{expansion_1}, as follows:
	\begin{align}\label{expansion1}
	u_\varepsilon(x) = \tilde{u}_0(x)  -\varepsilon\chi_{0,1}\left(\frac{x}{\varepsilon}\right)\cdot\nabla_{x}\tilde{u}_{0}(x)
	+ \sum_{l = 2}^{\theta}(-1)^l\varepsilon^{l}\chi_{0,l}\left(\frac{x}{\varepsilon}\right)\cdot \nabla_{x}^l\tilde{u}_0(x) + \mathcal{O}(\varepsilon^{\theta + 1}) 
	\end{align}
	with the cell functions $\chi_{0,l}$ for $1 \leq l \leq \theta$ satisfying the cell problems defined in \eqref{chi01} and \eqref{chigeneral}.

	At this point, we have derived the structure of two-scale asymptotic expansions where the scaling parameter $\alpha$ is positive. In the following, we show that the speed of convergence can be accelerated if the high-order asymptotic expansion is chosen appropriately. In addition, this questions how much regularity on the involved data we require to achieve the desired order of expansion as well as the rate of convergence.
	
	We introduce a smooth cut-off function $m^{\varepsilon} \in \mathcal{D}(\Omega)$ such that $0\le m^{\varepsilon} \le 1$ with
	\[
	m^{\varepsilon}=\begin{cases}
	0 & \text{if dist}\left(x,\partial\Omega\right)\le\varepsilon,\\
	1 & \text{if dist}\left(x,\partial\Omega\right)\ge2\varepsilon,
	\end{cases}\quad\text{and}\quad\varepsilon\left|\nabla m^{\varepsilon}\right|\le C,
	\]
	for which the following helpful estimates hold (cf. e.g. \cite{CP99}):
	\begin{equation}\label{asp_m}
	\left\Vert 1-m^{\varepsilon}\right\Vert _{L^{2}\left(\Omega^{\varepsilon}\right)}\le C\varepsilon^{\frac{1}{2}},\quad\varepsilon\left\Vert \nabla m^{\varepsilon}\right\Vert _{L^{2}\left(\Omega^{\varepsilon}\right)}\le C\varepsilon^{\frac{1}{2}}.
	\end{equation}
	
	\begin{remark}
		The use of this cut-off function to prove the convergence rates is not only seen in elliptic problems that we have taken into consideration, but also can be  found in some particular multiscale models. Aside from our earlier works \cite{Khoa17,KM16}, this technique is applied in the works \cite{Schmuck2012,Khoa2019} for a nonlinear drift-reaction-diffusion model in a heterogeneous solid‐electrolyte composite and in \cite{Schmuck2017} in the context of phase field equations. Besides, we single out the survey \cite{Zhikov2016} and the work \cite{Suslina2013} for a concrete background of the so-called operator corrector estimates related to this approach.
		
	\end{remark}
	
	Given a natural number $\mu \in [0,\theta - 1]$, we define the function $\psi_{\varepsilon}$ by
	\begin{align}
	\psi_\varepsilon := u_\varepsilon - \left(\tilde{u}_0 +\sum_{(k,l) \in \mathcal{M}_{\alpha,\mu}}\varepsilon^{k(\alpha + 1) + l}u_{k,l}\right) - m^\varepsilon\sum_{(k,l) \in \mathcal{M}_{\alpha,\theta}\backslash \mathcal{M}_{\alpha,\mu}}\varepsilon^{k(\alpha + 1) + l}u_{k,l}.\nonumber
	\end{align}
	Observe that $\psi_\varepsilon$ can be decomposed further as
	\begin{align}\label{decomposition}
	\psi_\varepsilon = \underbrace{u_\varepsilon - \tilde{u}_0 - \sum_{(k,l) \in \mathcal{M}_{\alpha,\theta}}\varepsilon^{k(\alpha + 1) + l}u_{k,l}}_{:=\varphi_{\varepsilon}} + \underbrace{(1 - m^\varepsilon)\sum_{\substack{(k,l) \in \mathcal{M}_{\alpha,\theta}\backslash \mathcal{M}_{\alpha,\mu}}}\varepsilon^{k(\alpha + 1) + l}u_{k,l}}_{:=\sigma_{\varepsilon}}.
	\end{align}
	
	Now, we state our convergence result.
	
	\begin{theorem}\label{mainthm1}
		Assume $\left(\text{A}_{1}\right)$, $\left(\text{A}_{5}\right)$ and $\left(\text{A}_{6}\right)$
		hold. Furthermore, suppose that $f\in C^{\infty}\left(\Omega\right)$
		and let $\mathcal{M}_{\alpha,\theta}$ be defined as in \eqref{u_0l} for given
		parameters $\alpha>0$ and $2\le\theta\in\mathbb{N}$. Let $u_{\varepsilon}$
		and $\tilde{u}_{0}$ be unique weak solutions of the microscopic problem
		$\left(P_{\varepsilon}^{R}\right)$ and the limit problem \eqref{limit_construct1-1}, respectively.
		Let $u_{k,l}$ be defined as in \eqref{u_0l} for $\left(k,l\right)\in\mathcal{M}_{\alpha,\theta}$.
		Then, for any $\mu\in\left[0,\theta-1\right]$ the following high-order
		corrector estimate holds:
		\begin{align*}
		\left\Vert u_{\varepsilon}-\tilde{u}_{0}-\sum_{\left(k,l\right)\in\mathcal{M}_{\alpha,\mu}}\varepsilon^{k\left(\alpha+1\right)+l}u_{k,l}-m^{\varepsilon}\sum_{\left(k,l\right)\in\mathcal{M}_{\alpha,\theta}\backslash\mathcal{M}_{\alpha,\mu}}\varepsilon^{k\left(\alpha+1\right)+l}u_{k,l}\right\Vert _{V^{\varepsilon}}\\\leq C\left(\varepsilon^{\theta - 1 + \alpha} + \varepsilon^{\mu + 1/2}\right).
		\end{align*}
	\end{theorem}
	\begin{proof}
		From the auxiliary problems \eqref{auxi_0}--\eqref{auxi_6} and the operators defined in \eqref{operators}, one can deduce, after some rearrangements, the following equation for $\varphi_\varepsilon$ in \eqref{decomposition}, which we refer to as the first difference equation:  
		\begin{align}\label{11_1}
		\nabla\cdot\left(-\mathbf{A}_\varepsilon(x)\nabla \varphi_\varepsilon\right) &= \varepsilon^\alpha\mathcal{R}(u_\varepsilon) - \sum_{\substack{(k,l) \in \mathcal{M}_{\alpha,\theta}\\l \leq \theta - 2}} \varepsilon^{k(\alpha + 1) + l + \alpha}\bar{\mathcal{R}}(u_{k,l}) \\-\sum_{\substack{(k,l) \in \mathcal{M}_{\alpha,\theta}\\l = \theta - 1}} & \varepsilon^{k(\alpha + 1) + l}(\mathcal{A}_1u_{k,l+1} + \mathcal{A}_2u_{k,l}) \nonumber- \sum_{\substack{(k,l) \in \mathcal{M}_{\alpha,\theta}\\l = \theta}}\varepsilon^{k(\alpha + 1) + l}\mathcal{A}_2u_{k,l},
		\end{align}
		associated with the boundary condition at $\Gamma^\varepsilon$:
		\begin{align}\label{boundary_2}
		-\mathbf{A}_\varepsilon(x)\nabla \varphi_\varepsilon \cdot \text{n} &= -\mathbf{A}_\varepsilon(x)\nabla u_\varepsilon \cdot \text{n} - \mathbf{A}_\varepsilon(x)\sum_{\substack{(k,l) \in \mathcal{M}_{\alpha,\theta}\\l \leq \theta - 1}}\varepsilon^{k(\alpha + 1) + l}(\nabla_{x}u_{k,l} + \nabla_yu_{k,l+1})\cdot\text{n} \\&- \mathbf{A}_\varepsilon(x)\sum_{\substack{(k,l) \in \mathcal{M}_{\alpha,\theta}\\l = \theta}}\varepsilon^{k(\alpha + 1) + l }\nabla_{x}u_{k,l}\cdot \text{n}.\nonumber
		\end{align}
		
		From the auxiliary problem \eqref{auxi_6}, the first term and the second term of the right-hand side of \eqref{boundary_2} vanishes naturally on the micro-surface $\Gamma^{\varepsilon}$. Thus, it yields
		\begin{align}\label{12_1}
		-\mathbf{A}_\varepsilon(x)\nabla \varphi_\varepsilon \cdot \text{n} =  - \mathbf{A}_\varepsilon(x)\sum_{\substack{(k,l) \in \mathcal{M}_{\alpha,\theta}\\l = \theta}}\varepsilon^{k(\alpha + 1) + l }\nabla_{x}u_{k,l}\cdot \text{n}.
		\end{align}
		
		Multiplying \eqref{11_1} by a test function $\varphi \in {V}^\varepsilon$, integrating the resulting equation by parts and then using the boundary information \eqref{12_1} together with the zero Dirichlet exterior condition, we get
		\begin{align}\label{intebypart}
		a(\varphi_\varepsilon,\varphi) &= \underbrace{\left\langle \varepsilon^\alpha \mathcal{R}(u_\varepsilon) - \sum_{\substack{
					\left(k,l\right)\in \mathcal{M}_{\alpha,\theta},
					l\le \theta - 2
			}}\varepsilon^{k(\alpha + 1) + l + \alpha}\bar{\mathcal{R}}(u_{k,l}),\varphi\right\rangle _{L^{2}(\Omega^\varepsilon)}}_{:=\mathcal{I}_{1}}\\&- \underbrace{\sum_{\substack{(k,l) \in \mathcal{M}_{\alpha,\theta},l = \theta - 1}}\varepsilon^{k(\alpha + 1) + l}\left\langle\mathcal{A}_1u_{k,l+1} + \mathcal{A}_2u_{k,l},\varphi\right\rangle_{L^2(\Omega^\varepsilon)}}_{:=\mathcal{I}_{2}} \nonumber\\&- \underbrace{\sum_{\substack{(k,l) \in \mathcal{M}_{\alpha,\theta},l = \theta}}\varepsilon^{k(\alpha + 1) + l}\left\langle\mathcal{A}_2u_{k,l},\varphi\right\rangle_{L^2(\Omega^\varepsilon)}}_{:=\mathcal{I}_{3}}\nonumber\\& + \underbrace{\int_{\Gamma^\varepsilon}\sum_{\substack{(k,l) \in \mathcal{M}_{\alpha,\theta}}}\varepsilon^{k(\alpha + 1) + l }\mathbf{A}_\varepsilon(x)\nabla u_{k,l}\cdot \text{n}\varphi dS_\varepsilon}_{:=\mathcal{I}_{4}},\nonumber
		\end{align}
		where $a:V^\varepsilon \times V^\varepsilon \to \mathbb{R}$ is the bilinear form defined in \eqref{bilinearform}.
		
		In order to find the upper bound of $\psi_{\varepsilon}$, we need to estimate from above all terms on the right-hand side of \eqref{intebypart}. We begin with the estimate for $\mathcal{I}_{1}$ by the following structural inequality: 
		\begin{align}\label{eq:Lipsin}
		\|\bar{\mathcal{R}}(u_{k,l})\|_{L^2(\Omega^\varepsilon)} \leq \mathbf{L}_\mathcal{R}\|u_{k,l}\|_{L^2(\Omega^\varepsilon)} + \|\bar{\mathcal{R}}(0)\|_{L^2(\Omega^\varepsilon)} \text{ for all } (k,l) \in \mathcal{M}_{\alpha,\theta},
		\end{align}
		by virtue of the globally Lipschitz function $\bar{\mathcal{R}}$. Due to \eqref{eq:Lipsin},  one can estimate from above $\mathcal{I}_{1}$ by
		\begin{align}\label{I1}
		\left|\mathcal{I}_1\right| &\leq \mathbf{L}_\mathcal{R}\sum_{\substack{(k,l) \in \mathcal{M}_{\alpha,\theta}\\l = \theta-1}}\varepsilon^{k(\alpha+1) + l + \alpha}(\|u_{k,l}\|_{L^2(\Omega^\varepsilon)} + \|\bar{\mathcal{R}}(0)\|_{L^2(\Omega^\varepsilon)})\|\varphi\|_{L^2(\Omega^\varepsilon)} \\&
		+ \mathbf{L}_\mathcal{R}\sum_{\substack{(k,l)\in \mathcal{M}_{\alpha,\theta}\\l = \theta}}\varepsilon^{k(\alpha + 1) + l + \alpha}(\|u_{k,l}\|_{L^2(\Omega^\varepsilon)} + \|\bar{\mathcal{R}}(0)\|_{L^2(\Omega^\varepsilon)})\|\varphi\|_{L^2(\Omega^\varepsilon)} \nonumber \\&\leq {C}\varepsilon^{\theta-1 + \alpha}\|\varphi\|_{V^\varepsilon}.\nonumber
		\end{align} 
		
		By the definition of the operators $\mathcal{A}_{1}$ and $\mathcal{A}_{2}$, we get for $(k,l)\in \mathcal{M}_{\alpha,\theta}$,
		\begin{align}
		\mathcal{A}_1 u_{k,l+1} &= \begin{cases}\label{A1kl} 0 &\text{if } k \neq 0,\\
		(-1)^{l+1}\left[-\mathbf{A}(y)\nabla_{y}\chi_{0,l+1}(y) +  \nabla_{y}(-\mathbf{A}(y)\chi_{0,l}(y))\right]\cdot\nabla_{x}^{l+2}\tilde{u}_0(x) &\text{if } k=0.
		\end{cases}\\
		\mathcal{A}_2u_{k,l} &=\begin{cases}\label{A2kl} 0 &\text{if } k \neq 0,\\
		(-1)^{l+1}\mathbf{A}(y)\chi_{0,l}(y)\nabla_x^{l+2}\tilde{u}_0(x) &\text{if } k = 0.
		\end{cases}
		\end{align}
		
		As stated in Remark \ref{remarkCinf}, we only need the source $f$ to be very smooth, says $f\in C^{\infty}(\Omega)$, to guarantee the uniform bound (with respect to $\varepsilon$) of all the involved derivatives of $\tilde{u}_0$ in \eqref{A1kl} and \eqref{A2kl}. We combine this with the fact that $\chi_{k,l} \in H_{\#}^1(Y_l)/\mathbb{R}$ and the assumptions $(\text{A}_1)$ and $P_{k,l} \in W^{2,\infty}(\Omega)$ for all $(k,l) \in \mathcal{M}_{\alpha,\theta}$ to get
		\begin{align}\label{I2I3}
		|\mathcal{I}_2 +\mathcal{I}_3| \leq C\varepsilon^{\theta-1 + \alpha}\|\varphi\|_{V^\varepsilon} \text{ for all } \varphi \in V^\varepsilon,
		\end{align}
		where we have used the Poincar\'{e} inequality.

		To estimate $\mathcal{I}_4$, we note that by the change of variables $x = \varepsilon y$, the following estimate holds
		\begin{align}
		\int_{\Gamma^\varepsilon}\left|\chi_{0,\theta}\left(\frac{x}{\varepsilon}\right)\nabla_{x}^{\theta + 1}\tilde{u}_0(x) \cdot \text{n}\right|^2dS_\varepsilon \leq C \varepsilon^{d - 1}\int_{\Gamma}|\chi_{0,\theta}(y)|^2dS_y.\nonumber
		\end{align}
		
		Since $|\Omega| \ge \varepsilon^d|Y|$ (due to our choice of perforated domains that $\varepsilon Y \subset \Omega $) together with the fact that the trace inequality in $Y_l$ is uniform with respect to $\varepsilon$, we estimate the above inequality as
		\begin{align}\label{est}
		\int_{\Gamma^\varepsilon}\left|\chi_{0,\theta}\left(\frac{x}{\varepsilon}\right)\nabla_{x}^{\theta + 1}\tilde{u}_0(x) \cdot \text{n}\right|^2dS_\varepsilon 
		\leq C \varepsilon^{ - 1}\|\chi_{0,\theta}\|_{H^1(Y_l)}^2
		\leq C\varepsilon^{ - 1}.
		\end{align}
		
		Combining \eqref{est} with assumption $(\text{A}_1)$, the trace inequality (cf. \cite[Lemma 2.31]{ADN59}) for $\Gamma^\varepsilon$ and the Poincar\'{e} inequality, we obtain
		\begin{align}\label{I4}
		|\mathcal{I}_4| \leq \varepsilon^\theta\| \chi_{0,\theta}\nabla_{x}^{\theta + 1}\tilde{u}_0 \cdot \text{n}\|_{L^2(\Gamma^\varepsilon)}\|\varphi\|_{L^2(\Gamma^\varepsilon)} 
		\leq C\varepsilon^{\theta - 1}\|\varphi\|_{V^\varepsilon}.
		\end{align}
		
		It now remains to estimate the second part $\sigma_\varepsilon$ of the decomposition \eqref{decomposition}. Similar to the above estimates of $\varphi_\varepsilon$, we consider the following quantity $\langle\sigma_\varepsilon,\varphi\rangle_{V^\varepsilon}$ for any $\varphi \in V^\varepsilon$. Observe that by the definition of $V^\varepsilon$ and by using the simple chain rule of differentiation, the estimate for $\sigma_\varepsilon$ is given by
		\begin{align}\label{est_sigma}
		&\left|\left\langle (1 - m^\varepsilon)\sum_{(k,l)\in \mathcal{M}_{\alpha,\theta}\backslash\mathcal{M}_{\alpha,\mu}}\varepsilon^{k(\alpha + 1)+l}u_{k,l},\varphi\right\rangle_{V^\varepsilon}\right| \\
		&\leq C\sum_{(k,l) \in \mathcal{M}_{\alpha,\theta}\backslash\mathcal{M}_{\alpha,\mu}}\varepsilon^{k(\alpha + 1) + l}\|\nabla(1-m^\varepsilon)\|_{L^2(\Omega^\varepsilon)}\|\varphi\|_{V^\varepsilon}\nonumber\\&+C\sum_{(k,l) \in \mathcal{M}_{\alpha,\theta}\backslash\mathcal{M}_{\alpha,\mu}}\varepsilon^{k(\alpha + 1) + l}\|1-m^\varepsilon\|_{L^2(\Omega^\varepsilon)}\|\varphi\|_{V^\varepsilon}\nonumber\\
		& \leq C\sum_{(k,l) \in \mathcal{M}_{\alpha,\theta}\backslash\mathcal{M}_{\alpha,\mu}}(\varepsilon^{k(\alpha + 1) + l - 1/2} + \varepsilon^{k(\alpha+1) + l + 1/2})\|\varphi\|_{V^\varepsilon}.\nonumber
		\end{align}
		
		Consequently, we finalize the estimate in \eqref{est_sigma} by
		\begin{align}\label{sigma_eps}
		\left|\left\langle (1 - m^\varepsilon)\sum_{(k,l)\in \mathcal{M}_{\alpha,\theta}\backslash\mathcal{M}_{\alpha,\mu}}\varepsilon^{k(\alpha + 1)+l}u_{k,l},\varphi\right\rangle_{V^\varepsilon}\right|
		\leq C(\varepsilon^{\mu + 1/2} + \varepsilon^{\mu + 3/2})\|\varphi\|_{V^\varepsilon}.
		\end{align}
		
		Thanks to the triangle inequality, we combine \eqref{intebypart}, \eqref{I1}, \eqref{I2I3}, \eqref{I4} and \eqref{sigma_eps} to get
		\begin{align}
		|\langle\psi_{\varepsilon},\varphi\rangle_{V^\varepsilon}| \leq C(\varepsilon^{\theta - 1 + \alpha} + \varepsilon^{\mu + 1/2} + \varepsilon^{\mu + 3/2})\|\varphi\|_{V^\varepsilon} \text{ for any } \varphi \in V^\varepsilon.\nonumber
		\end{align}
		
		By choosing $\varphi = \psi_{\varepsilon}$ and then by simplifying both sides of the resulting estimate, we complete the proof of the theorem.
	\end{proof}
	

	\subsubsection*{The case $\alpha < 0$}
	
	Recalling Theorem \ref{mainthm-2},  we have $u_\varepsilon \in \widetilde{\mathcal{W}}_\varepsilon$ for each $\varepsilon>0$, i.e. $
	\|u_\varepsilon\|_{\widetilde{\mathcal{W}}_\varepsilon}^2 \leq C.
	$ Note that our the underlying problem $(P_\varepsilon^R)$ is associated with the zero Neumann boundary condition on the micro-surfaces. 
	Thus, the structure of $\widetilde{\mathcal{W}}_\varepsilon$-norm reduces to
	\begin{align}\label{normw}
	\|u_\varepsilon\|_{\widetilde{\mathcal{W}}_\varepsilon}^2 = \|\nabla u_\varepsilon\|_{L^2(\Omega^\varepsilon)}^2 + \varepsilon^\alpha \|u_\varepsilon\|_{L^2(\Omega^\varepsilon)}^2 \leq C.
	\end{align}
	
	As a result of \eqref{normw}, we get
	\begin{align}\label{Onomega}
	\|u_\varepsilon\|_{L^2(\Omega^\varepsilon)} \leq C \varepsilon^{-\frac{\alpha}{2}},
	\end{align}
	which proves the strong convergence in $L^2(\Omega^\varepsilon)$ of $u_\varepsilon$ to zero as $\alpha < 0$ and $\varepsilon \searrow 0^{+}$.
	
	Moreover, by using the trace inequality for hypersurfaces $\Gamma^\varepsilon$ (cf. \cite[Lemma 3]{HJ91}), which reads as
	\begin{align}\label{traceinq}
	\varepsilon\|u_\varepsilon\|_{L^2(\Gamma^\varepsilon)}^2 \leq C\left(\|u_\varepsilon\|_{L^2(\Omega^\varepsilon)}^2 + \varepsilon^2\|\nabla u_\varepsilon\|_{L^2(\Omega^\varepsilon)}^2\right) \text{ for any } \varepsilon>0,
	\end{align}
	we combine \eqref{Onomega} with the fact that $\|\nabla u_\varepsilon\|_{L^2(\Omega^\varepsilon)}^2 \leq C$ from \eqref{normw}
	to obtain
	\begin{align}\label{onboundary_1}
	\varepsilon \|u_\varepsilon\|_{L^2(\Gamma^\varepsilon)}^2 
	\leq C\left(\varepsilon^{-\alpha} + \varepsilon^2\right).
	\end{align}
	
	Consequently, it follows from \eqref{onboundary_1} that
	\begin{align}\label{onboundary}
	\|u_\varepsilon\|_{L^2(\Gamma^\varepsilon)} \leq C\max\left\{\varepsilon^{-\frac{(\alpha + 1)}{2}},\varepsilon^{\frac{1}{2}}\right\} \text{ for any } \varepsilon>0.
	\end{align}

	In conclusion, combining \eqref{Onomega} and \eqref{onboundary_1} we claim the following theorem for the limit behavior of solution to problem $(P_\varepsilon^R)$ in the case $\alpha<0$.
	\begin{theorem}\label{maintheo_3}
		Assume $(\text{A}_1)$--$(\text{A}_4)$ hold. Suppose that $f \in L^2(\Omega^\varepsilon)$ and $\alpha < 0$. Let $u_\varepsilon$ be a unique solution in  $ \widetilde{\mathcal{W}}_\varepsilon$ of the problem $(P_\varepsilon^R)$. Then it holds:
		\begin{align}
		\|u_\varepsilon\|_{L^2(\Omega^\varepsilon)} + \sqrt{\varepsilon}\|u_\varepsilon\|_{L^2(\Gamma^\varepsilon)} \leq C\left(\varepsilon^{-\frac{\alpha}{2}} + \varepsilon\right).\nonumber
		\end{align}
	\end{theorem}
	\begin{remark}\label{remarkhay}
		From \eqref{onboundary},
		$u_\varepsilon$ converges strongly to zero in $L^2(\Gamma^\varepsilon)$  when $\alpha < -1$ and $\varepsilon \searrow 0^{+}$. We also remark that the internal source $f$ in  Theorem \ref{maintheo_3} just belongs to $L^2(\Omega^\varepsilon)$, which is quite different from the very smoothness of $f$ in Theorem \ref{mainthm1}. It is because in Theorem \ref{mainthm1} we need the boundedness of all the high-order derivatives of $\tilde{u}_0$ that solves \eqref{limit_construct1-1}, while 
		the linearization in Section \ref{sec:3} only requires $f \in L^2(\Omega^\varepsilon)$ to fulfill the estimate \eqref{normw} by the Lax--Milgram-based argument. 	
	\end{remark}
	\subsubsection{Surface reaction}\label{section4.1.3}
	For the problem $(P_{\varepsilon}^{S})$, we can proceed as in \cite{CP99-1}. If $\beta > 1$ we consider the following asymptotic expansion: 
	\begin{align*}
	u_\varepsilon(x) = u_0(x,y) + \varepsilon^{\beta - 1}u_{1,-1}(x,y) + \sum_{(k,l) \in \mathcal{Q}_{\beta,\theta}}\varepsilon^{k\beta + l}u_{k,l}(x,y) + \mathcal{O}(\varepsilon^{\theta + 1}),
	\end{align*}
	where $x \in \Omega^\varepsilon, y \in Y_l$, $u_{k,l}$ are periodic in $y$ and for $2\le \theta \in \mathbb{N}$, we define
	\begin{align}\label{setQ}
	\mathcal{Q}_{\beta,\theta}:= \{(k,l) \in [0,\theta]: k\beta + l \geq 1 \text{ and } k+l \leq \theta \}.
	\end{align}
	
	Taking assumptions on $\mathcal{S}$ as in $(\text{A}_5)$--$(\text{A}_6)$ and the fact that $u_{k,l}$ can be obtained by a family of linear partial differential equations for $(k,l) \in \mathcal{Q}_{\beta,\theta}$,  we thus state the following result.
	
	\begin{theorem}\label{maintheo_4}
		Assume that $(\text{A}_1)$ holds. Suppose that $f\in C^\infty(\Omega)$ and let $\mathcal{Q}_{\beta,\theta}$ be defined in \eqref{setQ} for given parameters $\beta>1$ and $2 \leq \theta \in \mathbb{N}$. Let $u_\varepsilon$ and $\tilde{u}_0$ be unique weak solutions of the microscopic problem $(P_\varepsilon^R)$ and the limit problem \eqref{limit_construct1-1}, respectively. For any $\mu \in [0,\theta - 1]$ the following higher-order corrector estimate holds
		\begin{align*}
		\left\|u_\varepsilon - \tilde{u}_0 - \sum_{(k,l)\in \mathcal{Q}_{\beta,\mu}}\varepsilon^{k\beta + l}u_{k,l} - m^{\varepsilon} \sum_{(k,l)\in \mathcal{Q}_{\beta,\theta}\backslash\mathcal{M}_{\beta,\mu}}\varepsilon^{k\beta + l}u_{k,l}\right\|_{V^\varepsilon} 
		\\\leq C\left(\varepsilon^{\theta + \beta - 2} + \varepsilon^{\mu + 1/2}\right).
		\end{align*}
	\end{theorem}
	One has immediately, by the same argument as in \eqref{Onomega}, that $u_{\varepsilon}$ converges strongly in $L^2(\Omega^{\varepsilon})$ to zero if $\beta <0$ and hence, a similar result to Theorem \ref{maintheo_3} can be obtained. Moreover, it remains to derive the convergence for $0\le \beta < 1$. If for any $f_1$, $f_2 \in H^1(\Omega^{\varepsilon})\cap L^{\infty}(\Omega^{\varepsilon})$ with $0 < \underline{c} \le f_{1},f_{2}\le \overline{c} <\infty$, we can find a function $f_3 \in L^{\infty}(\Omega^{\varepsilon})$ such that
	\begin{align}\label{eq:test1}
	\int_{\Omega^{\varepsilon}}f_{1}u_{\varepsilon}dx=\int_{\Gamma^{\varepsilon}}f_{2}\mathcal{S}\left(u_{\varepsilon}\right)dS_{\varepsilon}+\varepsilon f_{3},
	\end{align}
	then one can prove that (cf. \cite[Lemma 3.4]{KMnew}) for any $\varphi \in H^1(\Omega^{\varepsilon})$,
	\begin{align}\label{eq:test2}
	\left|\int_{\Omega^{\varepsilon}}f_{1}u_{\varepsilon}\varphi dx-\varepsilon\int_{\Gamma^{\varepsilon}}f_{2}\mathcal{S}\left(u_{\varepsilon}\right)\varphi dS_{\varepsilon}\right|\le C\varepsilon\left(\left\Vert \varphi\right\Vert _{H^{1}\left(\Omega^{\varepsilon}\right)}+\left\Vert f_{3}\right\Vert _{L^{\infty}\left(\Omega^{\varepsilon}\right)}\right).
	\end{align}
	
	Using \eqref{eq:test1} as an assumption, we state the following theorem.
	
	\begin{theorem}\label{maintheo_5}
		Assume $(\text{A}_1)$--$(\text{A}_4)$ and \eqref{eq:test1} hold. Suppose that $f \in L^2(\Omega^\varepsilon)$ and $\beta < 1$. Let $u_\varepsilon$ be a unique solution in  $ \widetilde{\mathcal{W}}_\varepsilon$ of the problem $(P_\varepsilon^S)$. Then it holds:
		\begin{align}
		\|u_\varepsilon\|_{L^2(\Omega^\varepsilon)} + \sqrt{\varepsilon}\|u_\varepsilon\|_{L^2(\Gamma^\varepsilon)} \leq C\left(\varepsilon^{\frac{1-\beta}{2}} + \varepsilon^{\frac{1}{2}} + \varepsilon\right).\nonumber
		\end{align}
	\end{theorem}
	\begin{proof}
		By a simple decomposition with the choice $\varphi = u_{\varepsilon}$, one thus has
		\begin{align}\label{eq:apchot}
		\int_{\Omega^{\varepsilon}}f_{1}u_{\varepsilon}^{2}dx & =\varepsilon\int_{\Gamma^{\varepsilon}}f_{2}\mathcal{S}\left(u_{\varepsilon}\right)u_{\varepsilon}dS_{\varepsilon}+\int_{\Omega^{\varepsilon}}f_{1}u_{\varepsilon}^{2}dx-\varepsilon\int_{\Gamma^{\varepsilon}}f_{2}\mathcal{S}\left(u_{\varepsilon}\right)u_{\varepsilon}dS_{\varepsilon} \\
		& \le\varepsilon\int_{\Gamma^{\varepsilon}}f_{2}\mathcal{S}\left(u_{\varepsilon}\right)u_{\varepsilon}dS_{\varepsilon}+C\varepsilon\left(\left\Vert u_{\varepsilon}\right\Vert _{H^{1}\left(\Omega^{\varepsilon}\right)}+\left\Vert f_{3}\right\Vert _{L^{\infty}\left(\Omega^{\varepsilon}\right)}\right).\nonumber
		\end{align}
		
		We turn our attention to the weak formulation for $(P_{\varepsilon}^{S})$, which reads as
		\[
		\int_{\Omega^{\varepsilon}}\mathbf{A}_{\varepsilon}\left(x\right)\nabla u_{\varepsilon}\cdot\nabla\varphi dx+\varepsilon^{\beta}\int_{\Gamma^{\varepsilon}}\mathcal{S}\left(u_{\varepsilon}\right)\varphi dS_{\varepsilon}=\int_{\Omega^{\varepsilon}}f\varphi dx\quad\text{for any }\varphi\in V^{\varepsilon}.
		\]
		
		Therefore, by choosing $\varphi = u_{\varepsilon}$ and  ($\text{A}_1$), one can estimate that
		\begin{align}\label{eq:apchot-1}
		\varepsilon\int_{\Gamma^{\varepsilon}}f_{2}\mathcal{S}\left(u_{\varepsilon}\right)u_{\varepsilon}dS_{\varepsilon}  \le C\varepsilon^{1-\beta}\left(\left\Vert f\right\Vert _{L^{2}\left(\Omega^{\varepsilon}\right)}\left\Vert u_{\varepsilon}\right\Vert _{L^{2}\left(\Omega^{\varepsilon}\right)}+\overline{\gamma}\left\Vert u_{\varepsilon}\right\Vert _{H^{1}\left(\Omega^{\varepsilon}\right)}^{2}\right).
		\end{align}
		
		Combining \eqref{eq:apchot} and \eqref{eq:apchot-1} and thanks to the trace inequality \eqref{traceinq}, we obtain the corrector result for $(P_{\varepsilon}^{S})$.
	\end{proof}

	\subsection{Volume-Surfaces reactions}\label{sec:4.2}
	This part is devoted to tackling the pore-scale elliptic problem $(P_{\varepsilon})$, based upon the analysis we have done above. Let us start off with the case $\alpha > 0, \beta > 1$ and consider
	\begin{align}\label{eq:newexpansion}
	u_\varepsilon(x) = u_0(x,y) + \varepsilon^\alpha u_{1,0,-1}(x,y) + \varepsilon^{\beta - 1}u_{0,1,-1}(x,y) \\+ \sum_{(k,l,n) \in \mathcal{M}_\theta}\varepsilon^{k(\alpha + 1) + l\beta + n}u_{k,l,n}(x,y) + \mathcal{O}(\varepsilon^{\theta + 1}),\nonumber
	\nonumber
	\end{align}
	where $x \in \Omega^\varepsilon, y \in Y_l$ and all components $u_{k,l,n}$ are periodic in $y$. For $\theta \ge 2$ we define the set $\mathcal{M}_{\theta}$ as
	\begin{align}\label{M_theta1}
	\mathcal{M}_\theta := \left\{(k,l,n) \in [0,\theta]: k(\alpha + 1) + l\beta + n \geq 1 \;\text{and}\; k + l + n \leq \theta \right\},
	\end{align}
	inspired very much by \eqref{Mo} and \eqref{setQ}. Moreover, we assume there exist Lipschitz-continuous functions $\bar{\mathcal{R}}$ and $\bar{\mathcal{S}}$ such that
	\begin{align*}
	\mathcal{R}(u_\varepsilon) &= \bar{\mathcal{R}}(u_0) + \varepsilon^{\alpha}\bar{\mathcal{R}}(u_{1,0,-1}) + \varepsilon^{\beta - 1}\bar{\mathcal{R}}(u_{0,1,-1}) \\ & + \sum_{(k,l,n) \in \mathcal{M}_\theta}\varepsilon^{k(\alpha + 1) + l\beta + n }\bar{\mathcal{R}}(u_{k,l,n}) + \mathcal{O}(\varepsilon^{\theta + 1}),\\
	\mathcal{S}(u_\varepsilon) &= \bar{\mathcal{S}}(u_0) + \varepsilon^{\alpha}\bar{\mathcal{S}}(u_{1,0,-1}) + \varepsilon^{\beta - 1}\bar{\mathcal{S}}(u_{0,1,-1})\\&+ \sum_{(k,l,n) \in \mathcal{M}_\theta}\varepsilon^{k(\alpha + 1) + l\beta + n}\bar{\mathcal{S}}(u_{k,l,n}) + \mathcal{O}(\varepsilon^{\theta + 1}).
	\end{align*}
	
	To avoid repeating cumbersome computations and unnecessary arguments, we only state the auxiliary problems and the limit system below, while the others are left to the interested reader. Using the convention in \eqref{operators}, the auxiliary problems are given by
	\begin{align}
	(\varepsilon^{-2}): &\begin{cases} \label{eq:auxx1}
	\mathcal{A}_0u_0 = 0 \quad \text{ in } Y_l,\\
	-\mathbf{A}(y)\nabla_{y}u_0\cdot \text{n} = 0 \quad \text{ on } \Gamma,\\
	u_0 \text{ is periodic in } y,
	\end{cases}\\
	(\varepsilon^{-1}): &\begin{cases} \label{eq:auxx2}
	\mathcal{A}_0u_{0,0,1} = -\mathcal{A}_1u_0 \quad \text{ in } Y_l,\\
	-\mathbf{A}(y)(\nabla_{x}u_0 + \nabla_{y}u_{0,0,1})\cdot \text{n}
	= 0 \quad \text{ on } \Gamma,\\
	u_{0,0,1} \text{ is periodic in } y,
	\end{cases}\\
	(\varepsilon^0): &\begin{cases} \label{eq:auxx3}
	\mathcal{A}_0u_{0,0,2} = f - \mathcal{A}_1u_{0,0,1} - \mathcal{A}_2u_0   \quad \text{ in } Y_l,\\
	-\mathbf{A}(y)(\nabla_{x}u_{0,0,1} + \nabla_{y}u_{0,0,2})\cdot \text{n} = 0 \quad \text{ on } \Gamma,\\
	u_{0,0,2} \text{ is periodic in } y,
	\end{cases}\\
	(\varepsilon^{\alpha - 2}): &\begin{cases} \label{eq:auxx4}
	\mathcal{A}_0 u_{1,0,-1} = 0 \quad \text{ in }Y_l,\\
	-\mathbf{A}(y)\nabla_{y}u_{1,0,-1}\cdot \text{n} = 0 \quad  \text{ on } \Gamma,\\
	u_{1,0,-1} \text{ is periodic in } y,
	\end{cases}\\
	(\varepsilon^{\alpha - 1}): &\begin{cases} \label{eq:auxx5}
	\mathcal{A}_0u_{1,0,0} = - \mathcal{A}_1 u_{1,0,-1} \quad \text{ in } Y_l,\\
	-\mathbf{A}(y)(\nabla_{x}u_{1,0,-1} + \nabla_{y}u_{1,0,0})\cdot \text{n} = 0 \text{ on }\Gamma,\\
	u_{1,0,0} \text{ is periodic in } y,
	\end{cases}\\
	(\varepsilon^\alpha): &\begin{cases} \label{eq:auxx6}
	\mathcal{A}_0u_{1,0,1} + \mathcal{R}(u_0) = - \mathcal{A}_1u_{1,0,0} - \mathcal{A}_2u_{1,0,-1}  \; \text{in } Y_l,\\
	-\mathbf{A}(y)(\nabla_{x}u_{1,0,0} + \nabla_{y}u_{1,0,1})\cdot \text{n} = 0 \quad \text{on } \Gamma,\\
	u_{1,0,1} \text{is periodic in } y,
	\end{cases}
	\end{align}
	\begin{align}
	(\varepsilon^{\beta - 3}): &\begin{cases} \label{eq:auxx7}
	\mathcal{A}_0u_{0,1,-1} = 0 \quad \text{ in } Y_l,\\
	-\mathbf{A}(y)\nabla_{y}u_{0,1,-1}\cdot \text{n} = 0 \quad  \text{ on } \Gamma,\\
	u_{0,1,-1} \text{ is periodic in } y,
	\end{cases}\\
	(\varepsilon^{\beta - 2}): &\begin{cases} \label{eq:auxx8}
	\mathcal{A}_0u_{0,1,0} = - \varepsilon^\alpha \mathcal{R}(u_{0,1,-1}) - \mathcal{A}_1u_{0,1,-1} \quad  \text{ in } Y_l,\\
	-\mathbf{A}(y)(\nabla_{x}u_{0,1,-1}  + \nabla_{y}u_{0,1,0})\cdot n =0 \text{ on } \Gamma,\\
	u_{0,1,0} \text{ is periodic in } y,
	\end{cases}\\
	(\varepsilon^{\beta - 1}): &\begin{cases} \label{eq:auxx9}
	\mathcal{A}_0u_{0,1,1} = - \mathcal{A}_1u_{0,1,0}-  \mathcal{A}_2u_{0,1,-1} \quad \text{ in } Y_l,\\
	-\mathbf{A}(y)(\nabla_{x}u_{0,1,0} + \nabla_{y}u_{0,1,1})\cdot \text{n} = \mathcal{S}(u_0) \quad \text{ on }\Gamma,\\
	u_{0,1,1} \text{ is periodic in } y,
	\end{cases}\\
	\vdots \nonumber \\
	(\varepsilon^{k(\alpha + 1) + l\beta + n}): &\begin{cases} \label{eq:auxx10}
	\mathcal{A}_0u_{k,l,n+2} = - \mathcal{A}_1u_{k,l,n+1} - \mathcal{A}_2u_{k,l,n}  \quad \text{in }Y_l,\\
	-\mathbf{A}(y)(\nabla_{x}u_{k,l,n+1} + \nabla_{y}u_{k,l,n + 2})\cdot \text{n} = 0 \quad \text{on } \Gamma,\\
	u_{k,l,n} \text{ is periodic in } y,
	\end{cases}
	\end{align}
	for all pairs $(k,l,n)\in \mathcal{K}_{\theta - 2}:= \mathcal{M}_{\theta - 2}\backslash \left\{(1,0,0);(0,1,0);(0,0,1)\right\}$.
	
	Once again, we obtain from \eqref{eq:auxx1} that
	$
	u_{0}(x,y) = \tilde{u}_0(x),
	$
	and hence the problems \eqref{eq:auxx2} and \eqref{eq:auxx3} are solvable in $u_{0,0,1}$ and $u_{0,0,2}$, respectively. On the other hand, from \eqref{eq:auxx4} and \eqref{eq:auxx7}, we may take $u_{1,0,-1}$ and $u_{0,1,-1}$ as zero functions, without loss of generality. From \eqref{eq:auxx5} and \eqref{eq:auxx8}, we have $u_{1,0,0}=u_{0,1,0}\equiv 0$ in accordance with $\left(\text{A}_3\right)$ and so is the function $u_{0,1,1}$ in \eqref{eq:auxx9}. Therefore, we conclude that the family of cell problems can be solved up to the high-order problems \eqref{eq:auxx10}. As a consequence, the corresponding cell problems can be obtained.
	
	From the auxiliary problems \eqref{eq:auxx1}--\eqref{eq:auxx3}, we know that
	$
	u_{0,0,1}(x,y) = -\chi_{0,0,1}(y)\cdot\nabla_{x}\tilde{u}_{0}(x),
	$
	where $\chi_{0,0,1}(y)$ is a field of cell functions whose cell problems are given by
	\begin{align}\label{chi001}
	\begin{cases}
	\mathcal{A}_{0}\chi_{0,0,1}^{j}=\dfrac{\partial\mathbf{A}_{ij}}{\partial y_{i}} & \text{in }Y_{l},\\
	-\mathbf{A}\left(y\right)\nabla_{y}\chi_{0,0,1}^{j}\cdot\text{n}=\mathbf{A}\left(y\right)\cdot\text{n}_j & \text{on }\Gamma,\\
	\chi_{0,0,1}^{j}\text{ is periodic},
	\end{cases}
	\end{align}
	which resembles the problem \eqref{chi01}. These cell problems admit a unique weak solution $\chi_{0,0,1}\in [H^{1}_{\#}(Y_{l})/ \mathbb{R}]^{d}$.
	
	As in \eqref{u_0l}, we get for $k = l = 0$ that
	$u_{0,0,n} = (-1)^{n}\chi_{0,0,n}(y)\cdot\nabla_{x}^{n}\tilde{u}_{0}(x)$,
	and then the high-order cell problems for this case are also determined, similar to \eqref{chigeneral}. In this context, the limit problem remains unchanged and can be recalled from \eqref{limit_construct1-1}--\eqref{limit_construct2-1} with the zero Dirichlet boundary condition for the exterior boundary.
	
	Let us now turn our attention to the corrector estimate in this case.
	
	\begin{theorem}\label{thmapcuoi}
		Assume $(A_1)$, $(A_5)$ and $(A_6)$ hold. Assume that $f \in C^\infty(\Omega)$ and let $\mathcal{M}_{\theta}$ be defined as in \eqref{M_theta1} for given parameters $\alpha > 0$, $\beta > 1$ and $2 \leq \theta \in \mathbb{N}$. Let $u_\varepsilon$ and $\tilde{u}_0$ be the unique weak solutions of the microscopic problem $(P_\varepsilon)$ and the limit problem \eqref{limit_construct1-1}, respectively. Let $u_{k,l,n}$ be solutions to the cell problems determined by the auxiliary problems \eqref{eq:auxx1}--\eqref{eq:auxx10} for $(k,l,n) \in \mathcal{M}_\theta$. Then, for any $\mu\in [0,\theta-1]$ the following high-order corrector estimate holds:
		\begin{align*}
		\left\|u_\varepsilon - \tilde{u}_0 - \sum_{(k,l,n) \in \mathcal{M}_\mu}\varepsilon^{k(\alpha + 1) + l\beta + n}u_{k,l,n} - m^\varepsilon\sum_{(k,l,n) \in \mathcal{M}_\theta\backslash\mathcal{M}_\mu}\varepsilon^{k(\alpha + 1) + l\beta + n}u_{k,l,n}\right\|_{V^\varepsilon} \nonumber\\\leq C(\varepsilon^{\theta - 1 + \alpha} + \varepsilon^{\theta + \beta - 2} + \varepsilon^{\mu + \min\{\mu\alpha,0\} + 1/2}).
		\end{align*}
	\end{theorem}
	\begin{proof}
		We set $\psi_{\varepsilon}:=\varphi_\varepsilon + \sigma_\varepsilon$, where
		\begin{align*}
		\varphi_\varepsilon &= u_\varepsilon - \tilde{u}_0 - \sum_{(k,l,n) \in \mathcal{M}_{\theta}}\varepsilon^{k(\alpha + 1) + l\beta + n}u_{k,l,n},\\\sigma_\varepsilon &= (1 - m^\varepsilon)\sum_{(k,l,n) \in \mathcal{M}_{\theta}\backslash \mathcal{M}_{\mu}}\varepsilon^{k(\alpha + 1) + l\beta + n}u_{k,l,n}.
		\end{align*}
		
		Therefore, we derive the difference equation for $\varphi_\varepsilon$ as in \eqref{11_1}, while the associated boundary condition is
		\begin{align*}
		-\mathbf{A}_\varepsilon(x) \nabla \varphi_\varepsilon \cdot \text{n} &= - \mathbf{A}_\varepsilon(x)\sum_{\substack{(k,l,n) \in \mathcal{M}_\theta\\n =\theta}}\varepsilon^{k(\alpha + 1) + l\beta + n}\nabla_{x}u_{k,l,n}\cdot \text{n} \nonumber\\&+ \varepsilon^\beta\bigg(\mathcal{S}(u_\varepsilon) - \sum_{\substack{(k,l,n) \in \mathcal{M}_\theta\\n \leq \theta - 2}}\varepsilon^{k(\alpha +1) + l\beta + n}\bar{\mathcal{S}}(u_{k,l,n})\bigg).
		\end{align*}
		
		For a test function $\varphi\in V^{\varepsilon}$, one can get the weak formulation of the difference equation for $\varphi_\varepsilon$, as follows:
		\begin{align}\label{intebypart_1}
		a(\varphi_\varepsilon,\varphi) &= \left\langle \varepsilon^\alpha \mathcal{R}(u_\varepsilon) - \sum_{\substack{
				\left(k,l,n\right)\in \mathcal{M}_{\theta},
				l\le \theta - 2
		}}\varepsilon^{k(\alpha + 1) + l\beta +n+ \beta}\bar{\mathcal{R}}(u_{k,l,n}),\varphi\right\rangle _{L^{2}(\Omega^\varepsilon)} \nonumber\\&- \sum_{\substack{(k,l,n) \in \mathcal{M}_{\theta},n = \theta - 1}}\varepsilon^{k(\alpha + 1) + l\beta + n}\left\langle\mathcal{A}_1u_{k,l,n+1} + \mathcal{A}_2u_{k,l,n},\varphi\right\rangle_{L^2(\Omega^\varepsilon)} \nonumber\\&- \sum_{\substack{(k,l,n) \in \mathcal{M}_{\theta},n = \theta}}\varepsilon^{k(\alpha + 1) + l\beta + n}\left\langle\mathcal{A}_2u_{k,l,n},\varphi\right\rangle_{L^2(\Omega^\varepsilon)}\nonumber\\& + \int_{\Gamma^\varepsilon}\sum_{\substack{(k,l,n) \in \mathcal{M}_{\theta}}}\varepsilon^{k(\alpha + 1) + l\beta +n }\mathbf{A}_\varepsilon(x)\nabla u_{k,l,n}\cdot \text{n}\varphi dS_\varepsilon \nonumber\\&+ \left\langle\varepsilon^\beta\bigg(\mathcal{S}(u_\varepsilon) - \sum_{\substack{(k,l,n) \in \mathcal{M}_\theta\\n \leq \theta - 2}}\varepsilon^{k(\alpha +1) + l\beta + n}\bar{\mathcal{S}}(u_{k,l,n})\bigg),\varphi\right\rangle_{L^2(\Gamma^\varepsilon)}.
		\end{align}
		
		At this stage, we observe that \eqref{intebypart_1} resembles \eqref{intebypart} except the last term on the right-hand side. Thus, it remains to estimate it from above. Clearly, using the Lipschitz property of $\mathcal{S}$ with the Poincar\'{e} inequality and the trace inequality on hypersurfaces, one gets
		\begin{align}\label{boundary_4}
		\left|\left\langle\varepsilon^\beta \mathcal{S}(u_\varepsilon) - \sum_{\substack{(k,l,n) \in \mathcal{M}_\theta\\n \leq \theta - 2}}\varepsilon^{k(\alpha + 1)+l\beta + n +\beta},\varphi\right\rangle_{L^2(\Gamma^\varepsilon)}\right| \leq C\varepsilon^{\theta + \beta - 2}\|\varphi\|_{V^\varepsilon}.
		\end{align}
		
		Following the same argument as for the estimate \eqref{sigma_eps}, we can bound the inner product of $\sigma_\varepsilon$ from above:
		\begin{align}\label{sigma_esp1}
		\left|\left\langle (1 - m^\varepsilon)\sum_{(k,l,n) \in \mathcal{M}_\theta \backslash \mathcal{M}_{\mu}}\varepsilon^{k(\alpha + 1) + l\beta + n}u_{k,l,n},\varphi\right\rangle\right| \nonumber\\\leq C(\varepsilon^{\mu +\min\{\mu\alpha,0\} + 1/2} + \varepsilon^{\mu + \min\{\mu\alpha,0\} + 3/2}).
		\end{align}
		
		Thanks to the triangle inequality. By choosing $\varphi = \psi_\varepsilon$, we combine \eqref{intebypart_1}, \eqref{boundary_4} and \eqref{sigma_esp1}  to get the corrector estimate
		\begin{align*}
		\|\psi_\varepsilon\|_{V^\varepsilon} \leq C(\varepsilon^{\theta - 1+\alpha} + \varepsilon^{\theta + \beta - 2} + \varepsilon^{\mu +\min\{\mu\alpha,0\} + 1/2} + \varepsilon^{\mu + \min\{\mu\alpha,0\} + 3/2} ).
		\end{align*}
		
		Hence, we complete the proof of the theorem.
	\end{proof}
	
	When either $\alpha<0$ or $\beta < 0$ is satisfied, the asymptotic limit of $u_{\varepsilon}$
	is close to zero. Indeed, we recall from Section \ref{sec:2} that
	\begin{equation}\label{eq:4.79}
	\left\Vert \nabla u_{\varepsilon}\right\Vert _{L^{2}\left(\Omega^{\varepsilon}\right)}^{2}+\left(\varepsilon^{\alpha}+\varepsilon^{\beta}\right)\left(\left\Vert u_{\varepsilon}\right\Vert _{L^{2}\left(\Omega^{\varepsilon}\right)}^{2}+\left\Vert u_{\varepsilon}\right\Vert _{L^{2}\left(\Gamma^{\varepsilon}\right)}^{2}\right)\le C.
	\end{equation}
	
	Thanks to the elementary Bunyakovsky--Cauchy--Schwarz inequality, we deduce from
	\eqref{eq:4.79} that
	\begin{align*}
	\varepsilon^{\frac{\alpha}{2}}\left\Vert u_{\varepsilon}\right\Vert _{L^{2}\left(\Omega^{\varepsilon}\right)}+\varepsilon^{\frac{\beta}{2}}\left\Vert u_{\varepsilon}\right\Vert _{L^{2}\left(\Gamma^{\varepsilon}\right)}\le C,\;
	\varepsilon^{\frac{\beta}{2}}\left\Vert u_{\varepsilon}\right\Vert _{L^{2}\left(\Omega^{\varepsilon}\right)}+\varepsilon^{\frac{\alpha}{2}}\left\Vert u_{\varepsilon}\right\Vert _{L^{2}\left(\Gamma^{\varepsilon}\right)}\le C.
	\end{align*}
	
	Henceforth, the theorems for these cases can be stated as in Theorem \ref{maintheo_3}.
	
	When $\alpha > 0$ and $0 \leq \beta < 1$, we proceed as in Subsection \ref{section4.1.3}: for any $f_1, f_2 \in H^1(\Omega^\varepsilon) \cap L^\infty(\Omega^\varepsilon)$ with $0 <\underline{c} \leq f_1, f_2\leq \bar{c}< \infty$, assume that we can find a function $f_3 \in L^\infty(\Omega^\varepsilon)$ such that \eqref{eq:test1} holds.
	Then we are led to the estimates \eqref{eq:test2} and \eqref{eq:apchot}. Recalling the weak formulation of $(P_\varepsilon)$, which reads as
	\begin{align*}
	\int_{\Omega^\varepsilon}\mathbf{A}_\varepsilon(x) \nabla u_\varepsilon \cdot \nabla \varphi dx + \varepsilon^\beta \int_{\Gamma^\varepsilon}\mathcal{S}(u_\varepsilon)\varphi d\mathcal{S}_\varepsilon + \varepsilon^\alpha\int_{\Omega^\varepsilon}\mathcal{R}(u_\varepsilon)\varphi dx = \int_{\Omega^\varepsilon}f\varphi dx \nonumber
	\end{align*}
	for any $ \varphi \in H^1(\Omega^\varepsilon)$, by choosing $\varphi = u_\varepsilon$ and $(\text{A}_1)$--$(\text{A}_4)$ one can get
	\begin{align}\label{lastcase2}
	\varepsilon\int_{\Gamma^\varepsilon}\mathcal{S}(u_\varepsilon)u_\varepsilon d\mathcal{S}_\varepsilon  \leq \varepsilon^{1-\beta}(\bar{\gamma}\|u_\varepsilon\|_{H^1(\Omega^\varepsilon)}^2 + \delta_1 \varepsilon^\alpha\|u_\varepsilon\|_{L^2(\Omega^\varepsilon)}^2 + \|f\|_{L^2(\Omega^\varepsilon)}\|u_\varepsilon\|_{L^2(\Omega^\varepsilon)}).
	\end{align}
	
	Combining \eqref{eq:apchot} and \eqref{lastcase2}, we obtain 
	\begin{align}\label{final3}
	\|u_\varepsilon\|_{L^2(\Omega^\varepsilon)}^2 \leq C(\varepsilon^{1-\beta}(1 + \varepsilon^\alpha) + \varepsilon).
	\end{align}
	
	Applying the trace inequality \eqref{traceinq} to \eqref{final3}, we can finalize the corrector results for $(P_\varepsilon)$ by the following theorem.
	\begin{theorem}\label{final_thm}
		Assume $(\text{A}_1)$--$(\text{A}_3)$ hold. Suppose that $f \in L^2(\Omega^\varepsilon)$ and $\alpha > 0, 0 \leq \beta <1$. Let $u_\varepsilon$ be a unique solution in $\widetilde{\mathcal{W}}_\varepsilon$ of the problem $(P_\varepsilon)$. Then, the following estimate holds:
		\begin{align}
		\|u_\varepsilon\|_{L^2(\Omega^\varepsilon)} + \sqrt{\varepsilon}\|u_\varepsilon\|_{L^2(\Gamma^\varepsilon)} \leq C\left(\varepsilon^{ \frac{1- \beta}{2}}(1 + \varepsilon^{\frac{\alpha}{2}}) + \varepsilon^\frac{1}{2} \right).
		\end{align}
	\end{theorem}
	\subsection{A numerical example}\label{sec:4.3}
	
	\begin{figure}[h]
		\centering
		\begin{subfigure}[b]{0.3\textwidth}
			\includegraphics[scale=0.25]{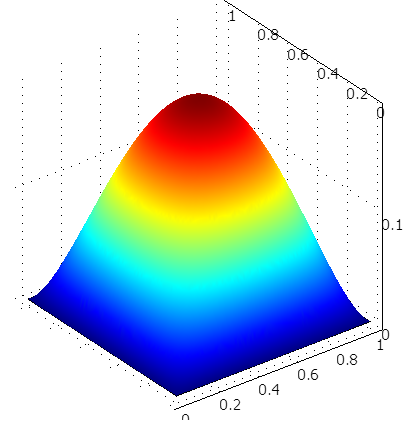}\includegraphics[scale=0.25]{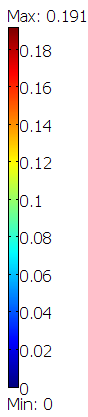}
			\caption{$\tilde{u}_0$}
			\label{Eu0}
		\end{subfigure}
		~$\qquad$
		\begin{subfigure}[b]{0.3\textwidth}
			\includegraphics[scale=0.18]{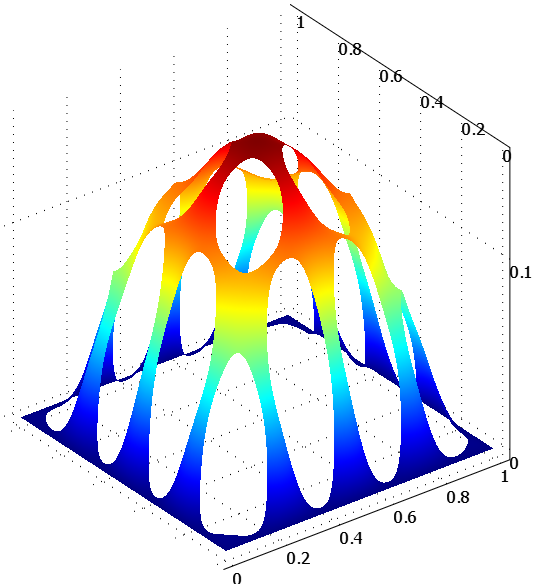}\includegraphics[scale=0.18]{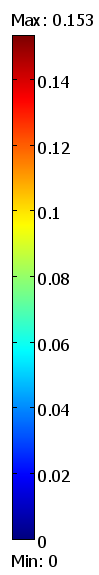}
			\caption{$\varepsilon=0.25$}
			\label{E025}
		\end{subfigure}
		\\
		\begin{subfigure}[b]{0.3\textwidth}
			\includegraphics[scale=0.25]{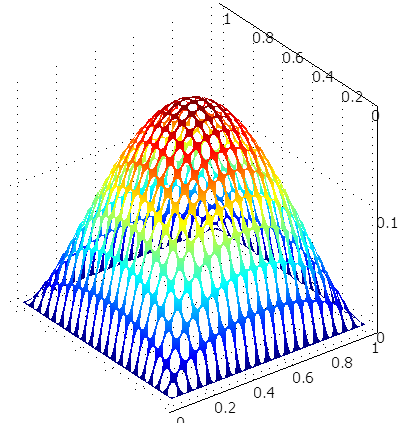}\includegraphics[scale=0.25]{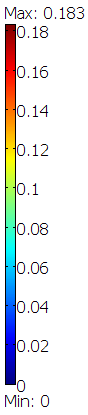}
			\caption{$\varepsilon=0.05$}
			\label{E005}
		\end{subfigure}
		~ $\qquad$
		\begin{subfigure}[b]{0.3\textwidth}
			\includegraphics[scale=0.25]{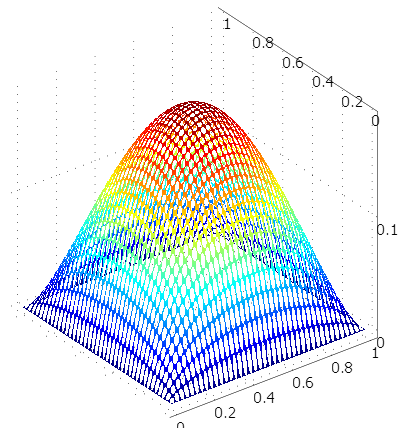}\includegraphics[scale=0.25]{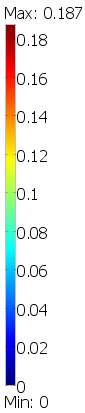}
			\caption{$\varepsilon=0.025$}
			\label{E0025}
		\end{subfigure}
		\caption{Comparison between the homogenized solution and the microscopic solution for $\varepsilon\in \left\{0.25,0.05,0.025\right\}$.}
		\label{compare1}
	\end{figure} 
	Here, we illustrate the asymptotic behaviors of the microscopic problem for different values of the scaling factors $\alpha,\beta\in\mathbb{R}$. For simplicity, we consider $(P_{\varepsilon})$ in two dimensions with the linear mappings $\mathcal{R},\mathcal{S}$ of the form $\mathcal{R}(z)=C_{1}z, \mathcal{S}(z)=C_{2}z$ for $C_1>0$ and $C_{2}\ge 0$. Taking $C_{1}=1$, we arrive at a modified Helmholtz-type equation. We choose the unit square $\Omega = (0,1)\times (0,1)$ the domain of interest and the oscillatory diffusion as
	\begin{align*}
	\mathbf{A}(x/\varepsilon) = \frac{1}{2 + \cos\left(\frac{2\pi x_{1}}{\varepsilon}\right)\cos\left(\frac{2\pi x_{2}}{\varepsilon}\right)}.
	\end{align*}
	Moreover, we consider the unit cell $Y=(0,1)\times (0,1)$ with a reference circular hole of radius $r = 0.4$ and take the volumetric source $f = 1$ and define the volume porosity as $\left|Y_{l}\right|=1 - \pi r^2 \approx 0.497$. According to \eqref{chi001}, the effective diffusion coefficient is computed as 
	\[
	\bar{\mathbf{A}}=\begin{pmatrix}0.191613 & 2.025\times10^{-9}\\
	2.025\times10^{-9} & 0.191613
	\end{pmatrix}.
	\]
	
	\begin{figure}[h]
		\centering
		\begin{subfigure}[b]{0.4\textwidth}
			\includegraphics[scale=0.424]{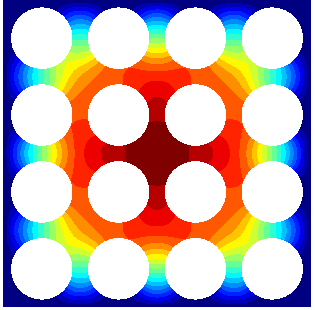}\includegraphics[scale=0.4]{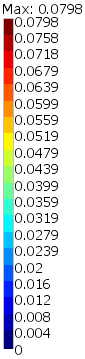}
			\caption{$\alpha=-1$ and $\beta=1$.}
		\end{subfigure}
		\hspace{0.48cm}
		\begin{subfigure}[b]{0.48\textwidth}
			\includegraphics[scale=0.424]{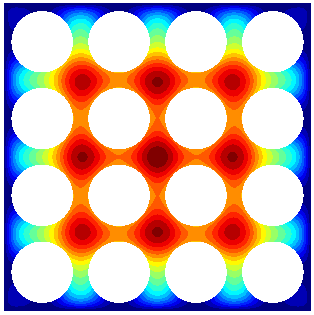}\includegraphics[scale=0.4]{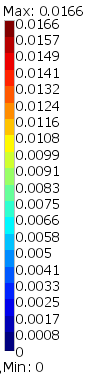}
			\caption{$\alpha=1$ and $\beta=-1$.}
		\end{subfigure}
		\\
		\begin{subfigure}[b]{0.4\textwidth}
			\includegraphics[scale=0.424]{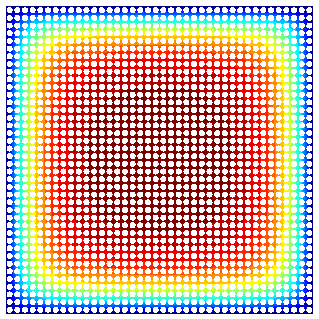}\includegraphics[scale=0.4]{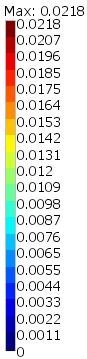}
			\caption{$\alpha=-1$ and $\beta=1$.}
		\end{subfigure}
		\hspace{0.55cm}
		\begin{subfigure}[b]{0.48\textwidth}
			\includegraphics[scale=0.425]{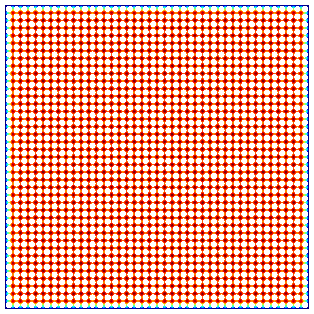}\includegraphics[scale=0.4]{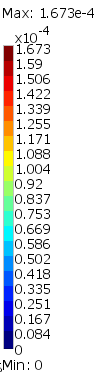}
			\caption{$\alpha=1$ and $\beta=-1$.}
		\end{subfigure}
		\caption{Behavior of the microscopic solution $u_{\varepsilon}$ for the sub-cases  $\alpha = -1, \beta = 1$ and $\alpha = 1, \beta = -1$ at $\varepsilon=0.25$ (top) and $\varepsilon=0.025$ (bottom).}
		\label{caseII}
	\end{figure}
	
	\begin{figure}[h]
		\includegraphics[scale = 0.28]{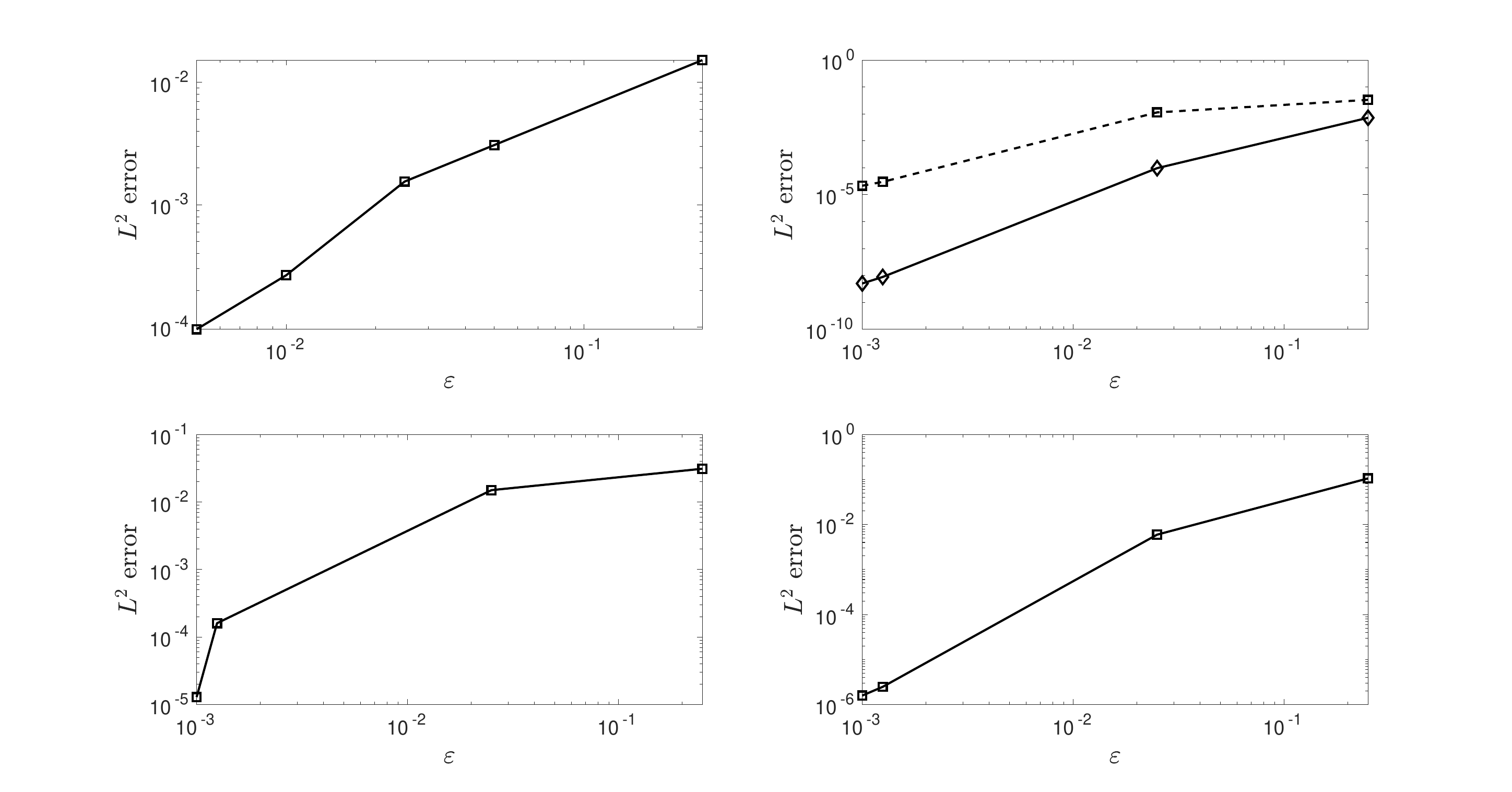}
		\caption{\small Convergence results in the $\ell^{2}$-norm of $u_{\varepsilon}$ in the microscopic domain for
			various combinations of the parameters $\alpha,\beta$ and choices of $\varepsilon$. {\em First panel:} $\alpha=1,\beta=2$. {\em Second panel:} $\alpha=-1,\beta=1$ (dashed square) and $\alpha=1,\beta=-1$ (solid diamond). {\em Third panel:} $\alpha=1,\beta=1/2$. {\em Fourth panel:} convergence at the micro-surfaces for $\alpha=-2,C_{2}=0$. }
		\label{fig:2a}
	\end{figure}

	\subsubsection*{Comments on numerical results}
	To verify our theoretical results, we divide the scale factors $\alpha$ and $\beta$ into three cases:
	\begin{enumerate}
		\item When $\alpha >0$ and $\beta > 1$, $u_{\varepsilon}$   converges to $\tilde{u}_{0}$ of the homogenized problem \eqref{limit_construct1-1}--\eqref{limit_construct2-1}.
		\item When either $\alpha < 0$ or $\beta < 0$, $u_{\varepsilon }$ converges to 0.
		\item When $\alpha >0$ and $0\le \beta <1$, $u_{\varepsilon }$ converges to 0.
	\end{enumerate}
	
	Suppose we take $C_{2}=1$. We consider the first case by fixing $\alpha = 1$ and $\beta = 2$. Here, we use the standard linear FEM with a mesh discretization, which is much more smaller than $\varepsilon$ to solve the microscopic problem for various values of $\varepsilon \in \left\{0.25,0.05,0.025,0.01,0.005\right\}$. In Figure \ref{compare1}, we compare the homogenized solution $\tilde{u}_{0}$ with the microscopic solution $u_{\varepsilon}$ at some chosen values of $\varepsilon \in \left\{0.25, 0.01, 0.025\right\}$. It can be seen that the microscopic solution converges to the homogenized solution as $\varepsilon$ tends to 0. This confirms the usual performance of our homogenization procedure. Based on Table \ref{table:1a}, it can also be seen that the homogenized solution $\tilde{u}_{0}$ is the excellent approximation candidate when $\varepsilon$ gets smaller and smaller.
	
	For the second case, we verify the sub-cases  $\alpha = -1, \beta = 1$ and $\alpha = 1, \beta = -1$, respectively. As depicted in Figure \ref{caseII}, we find that $u_{\varepsilon}$ converges to 0 as $\varepsilon \searrow 0^{+}$, which agrees with  Theorem \ref{maintheo_3}. Moreover, we have tabulated in Table \ref{table:1b} the smallness of the microscopic solution in $\ell^2$-norm of these cases at various  $\varepsilon\in\left\{0.25,0.025,0.00125,0.001\right\}$. 
	In the same spirit, choosing $\alpha = 1$ and $\beta = 1/2$ the convergence to 0 of $u_{\varepsilon}$ is guaranteed by  Theorem \ref{final_thm} and this is verified by the numerical results tabulated in Table \ref{table:1c}.
	
	It is worth mentioning that we can also corroborate the case $\alpha < - 1$ discussed in Remark \ref{remarkhay} where $u_{\varepsilon}$ converges to 0 at the micro-surfaces. Indeed, taking $C_2 = 0$ and $\alpha = -2$ we obtain the numerical results in Table \ref{table3}, which are consistent with Theorem \ref{maintheo_3}.
	
		The convergence rates in Tables \ref{table:1} and \ref{table3} are also depicted in Figure \ref{fig:2a}, where we show log-log plots of the numerical errors.
	
	\begin{table}
		\centering
		\begin{subtable}[t]{5in}
			\centering
			\begin{tabular}{|c|c|c|c|c|c|}
				\hline 
				$\varepsilon$ & 0.25 & 0.05 & 0.025 & 0.01 & 0.005\tabularnewline
				\hline 
				$\left\Vert u_{\varepsilon}-\tilde{u}_{0}\right\Vert _{\ell^{2}}$ & 0.015219 & 0.003079 & 0.001550 & 0.000266 & 9.623$\times10^{-5}$\tabularnewline
				\hline 
			\end{tabular}
			\caption{Case 1 ($\alpha=1,\beta=2$)}\label{table:1a}
		\end{subtable}
		\\
		\begin{subtable}[t]{4in}
			\centering
			\begin{tabular}{|c|c|c|c|c|}
				\hline 
				$\varepsilon$ & 0.25 & 0.025 & 0.00125 & 0.001\tabularnewline
				\hline 
				\multicolumn{5}{|c|}{$\alpha=-1,\beta=1$}\tabularnewline
				\hline 
				$\left\Vert u_{\varepsilon}\right\Vert _{\ell^{2}}$ & 0.0333 & 0.0114 & 2.9938$\times10^{-5}$ & 2.1528$\times10^{-5}$\tabularnewline
				\hline 
				\multicolumn{5}{|c|}{$\alpha=1,\beta=-1$}\tabularnewline
				\hline 
				$\left\Vert u_{\varepsilon}\right\Vert _{\ell^{2}}$ & 0.0072 & 9.6755$\times10^{-5}$ & 8.789$\times10^{-9}$ & 5.0318$\times10^{-9}$\tabularnewline
				\hline 
			\end{tabular}
			\caption{Case 2 ($\alpha=-1,\beta=1$ and $\alpha=1,\beta=-1$)}\label{table:1b}
		\end{subtable}
		\\
		\begin{subtable}[t]{4in}
			\centering
			\begin{tabular}{|c|c|c|c|c|}
				\hline 
				$\varepsilon$ & 0.25 & 0.025 & 0.00125 & 0.001\tabularnewline
				\hline 
				$\left\Vert u_{\varepsilon}\right\Vert _{\ell^{2}}$ & 0.0311 & 0.0150 & 1.6029$\times10^{-4}$ & 1.2887$\times10^{-5}$\tabularnewline
				\hline 
			\end{tabular}
			\caption{Case 3 ($\alpha=1,\beta=1/2$)}\label{table:1c}
		\end{subtable}
		\caption{Numerical results in the $\ell^{2}$-norm of $u_{\varepsilon}$ in the microscopic domain for
			various combinations of the parameters $\alpha,\beta$ and choices of $\varepsilon$.}\label{table:1}
	\end{table}
	
	
	\begin{table}
		\noindent \begin{centering}
			\begin{tabular}{|c|c|c|c|c|}
				\hline 
				$\varepsilon$ & 0.25 & 0.025 & 0.00125 & 0.001\tabularnewline
				\hline 
				$\left\Vert u_{\varepsilon}\right\Vert _{\ell^{2}\left(\Gamma^{\varepsilon}\right)}$ & 0.10589 & 0.00596 & 2.474$\times10^{-6}$ & 1.584$\times10^{-6}$\tabularnewline
				\hline 
			\end{tabular}
			\par\end{centering}
		\caption{Numerical results in the $\ell^{2}$-norm of $u_{\varepsilon}$ at the micro-surfaces for $\alpha=-2,C_{2}=0$. }\label{table3}
	\end{table}

	\subsection*{Acknowledgement}
	V.A.K acknowledges Mai Thanh Nhat Truong (Dongguk University,  Republic of Korea) for fruitful discussions on the graphical structures of porous media. V.A.K would like to thank Prof. Adrian Muntean (Karlstad University, Sweden) for the whole-hearted instruction in his PhD period at GSSI, Italy. T.K.T.T is truly grateful to Prof. Adrian Muntean for the supervision during the Ph.D. training at GSSI and Karlstad University.
	\begin{center}
		\bibliographystyle{plain}
		\bibliography{mybib}

\begin{thebibliography}{10}

\bibitem{ADN59}
S.~Agmon, A.~Douglis, and L.~Nirenberg.
\newblock Estimates near the boundary for solutions of elliptic partial
  differential equations satisfying general boundary value conditions {I}.
\newblock {\em Communications on Pure and Applied Mathematics}, 12:623--727,
  1959.

\bibitem{Allaire1999}
G.~Allaire and M.~Amar.
\newblock Boundary layer tails in periodic homogenization.
\newblock {\em {ESAIM}: Control, Optimisation and Calculus of Variations},
  4:209--243, 1999.

\bibitem{Armstrong2016}
S.~Armstrong, A.~Gloria, and T.~Kuusi.
\newblock Bounded correctors in almost periodic homogenization.
\newblock {\em Archive for Rational Mechanics and Analysis}, 222(1):393--426,
  2016.

\bibitem{CM02}
G.~A. Chechkin and T.~A. Mel\'nyk.
\newblock Asymptotics of eigenelements to spectral problem in thick cascade
  junction with concentrated masses.
\newblock {\em Applicable Analysis}, 91(6):1055--1095, 2012.

\bibitem{CP99-1}
G.~A. Chechkin and A.~L. Piatnitski.
\newblock Homogenization of boundary-value problem in a locally periodic
  perforated domain.
\newblock {\em Applicable Analysis}, 71(1):215--235, 1999.

\bibitem{CP99}
D.~Cioranescu and {J. Saint Jean Paulin}.
\newblock {\em Homogenization of {R}eticulated {S}tructures}.
\newblock Springer, 1999.

\bibitem{DHS16}
C.~D\"orlemann, M.~Heida, and B.~Schweizer.
\newblock Transmission conditions for the {H}elmholtz-{E}quation in perforated
  domains.
\newblock {\em Vietnam Journal of Mathematics}, 45(1--2):241--253, 2016.

\bibitem{Frank2011}
F.~Frank, N.~Ray, and P.~Knabner.
\newblock Numerical investigation of homogenized
  {S}tokes{\textendash}{N}ernst{\textendash}{P}lanck{\textendash}{P}oisson
  systems.
\newblock {\em Computing and Visualization in Science}, 14(8):385--400, 2011.

\bibitem{GMRL09}
J.~Garc\'ia-Meli\'an, J.~D. Rossi, and J.~C. Sabina~De Lis.
\newblock Existence and uniqueness of positive solutions to elliptic problems
  with sublinear mixed boundary conditions.
\newblock {\em Communications in Contemporary Mathematics}, 11(4):585--613,
  2009.

\bibitem{Gaudiello2018}
A.~Gaudiello and T.~Mel'nyk.
\newblock Homogenization of a nonlinear monotone problem with nonlinear
  {S}ignorini boundary conditions in a domain with highly rough boundary.
\newblock {\em Journal of Differential Equations}, 265(10):5419--5454, 2018.

\bibitem{GT83}
D.~Gilbarg and N.~Trudinger.
\newblock {\em Elliptic {P}artial {D}ifferential {E}quations of {S}econd
  {O}rder}.
\newblock Springer-Verlag, 1983.

\bibitem{Griso2004}
G.~Griso.
\newblock Error estimate and unfolding for periodic homogenization.
\newblock {\em Asymptotic Analysis}, 40:269--286, 2004.

\bibitem{HJ91}
U.~Hornung and W.~J\"ager.
\newblock {D}iffusion, convection, adsorption, and reaction of chemicals in
  porous media.
\newblock {\em Journal of Differential Equations}, 92:199--225, 1991.

\bibitem{Kacur99}
J.~Ka{c}ur.
\newblock Solution to strongly nonlinear parabolic problems by a linear
  approximation scheme.
\newblock {\em IMA Journal of Numerical Analysis}, 19:119--145, 1999.

\bibitem{Khoa17}
V.~A. Khoa.
\newblock A high-order corrector estimate for a semi-linear elliptic system in
  perforated domains.
\newblock {\em Comptes Rendus M\' ecanique}, 345(5):337--343, 2017.

\bibitem{KM16}
V.~A. Khoa and A.~Muntean.
\newblock Asymptotic analysis of a semi-linear elliptic system in perforated
  domains: Well-posedness and corrector for the homogenization limit.
\newblock {\em Journal of Mathematical Analysis and Applications},
  439:271--295, 2016.

\bibitem{KM16-1}
V.~A. Khoa and A.~Muntean.
\newblock A note on iterations-based derivations of high-order homogenization
  correctors for multiscale semi-linear elliptic equations.
\newblock {\em Applied Mathematics Letters}, 58:103--109, 2016.

\bibitem{KMnew}
V.~A. Khoa and A.~Muntean.
\newblock Correctors justification for a {S}moluchowski-{S}oret-{D}ufour model
  posed in perforated domains.
\newblock submitted (arXiv:1704.01790), 2017.

\bibitem{Khoa2019}
V.~A. Khoa and A.~Muntean.
\newblock Corrector homogenization estimates for a non-stationary
  {S}tokes--{N}ernst--{P}lanck--{P}oisson system in perforated domains.
\newblock {\em Communications in Mathematical Sciences}, 17(3):705--738, 2019.

\bibitem{Kim2018}
S.~Kim and K.-A. Lee.
\newblock Higher order convergence rates in theory of homogenization {III}:
  {V}iscous {H}amilton{\textendash}{J}acobi equations.
\newblock {\em Journal of Differential Equations}, 265(10):5384--5418, 2018.

\bibitem{KAM14}
O.~Krehel, T.~Aiki, and A.~Muntean.
\newblock Homogenization of a thermo-diffusion system with {S}moluchowski
  interactions.
\newblock {\em Networks and Heterogeneous Media}, 9(4):739--762, 2014.

\bibitem{KMK15}
O.~Krehel, A.~Muntean, and P.~Knabner.
\newblock Multiscale modeling of colloidal dynamics in porous media including
  aggregation and deposition.
\newblock {\em Advances in Water Resources}, 86:209--216, 2015.

\bibitem{LDD02}
N.~T. Long, A.~P.~N. Dinh, and T.~N. Diem.
\newblock Linear recursive schemes and asymptotic expansion associated with the
  {K}irchoff--{C}arrier operator.
\newblock {\em Journal of Mathematical Analysis and Applications},
  267(1):116--134, 2002.

\bibitem{Mel95}
T.~A. Melnik.
\newblock Asymptotic expansion of eigenvalues and eigenfunctions for elliptic
  boundary-value problems with rapidly oscillating coefficients in a perforated
  cube.
\newblock {\em Journal of Mathematical Sciences}, 75(3):1646--1671, 1995.

\bibitem{Muthukumar2009}
T.~Muthukumar and A.~K. Nandakumaran.
\newblock Homogenization of low-cost control problems on perforated domains.
\newblock {\em Journal of Mathematical Analysis and Applications},
  351(1):29--42, 2009.

\bibitem{Oleinik1992}
O.~A. Oleinik, A.~S. Shamaev, and G.~A. Yosifian.
\newblock {\em Mathematical {P}roblems in {E}lasticity and {H}omogenization}.
\newblock North Holland, 1992.

\bibitem{Onofrei2007}
D.~Onofrei and B.~Vernescu.
\newblock Error estimates in periodic homogenization with non-smooth
  coefficients.
\newblock {\em Asymptotic Analysis}, 54:103--123, 2007.

\bibitem{Onofrei2012}
D.~Onofrei and B.~Vernescu.
\newblock Asymptotic analysis of second-order boundary layer correctors.
\newblock {\em Applicable Analysis}, 91(6):1097--1110, 2012.

\bibitem{Pao93}
C.~V. Pao.
\newblock {\em Nonlinear Parabolic and Elliptic Equations}.
\newblock Springer, 1993.

\bibitem{Papanicolau1978}
G.~Papanicolau, A.~Bensoussan, and J.-L. Lions.
\newblock {\em Asymptotic {A}nalysis for {P}eriodic {S}tructures}.
\newblock North Holland, 1978.

\bibitem{ray2013thesis}
N.~Ray.
\newblock {\em Colloidal {T}ransport in {P}orous {M}edia {M}odeling and
  {A}nalysis}.
\newblock PhD thesis, University of Erlangen-Nuremberg, 2013.

\bibitem{RMK12}
N.~Ray, A.~Muntean, and P.~Knabner.
\newblock Rigorous homogenization of a {S}tokes-{N}ernst-{P}lanck-{P}oisson
  system.
\newblock {\em Journal of Mathematical Analysis and Applications},
  390(1):374--393, 2012.

\bibitem{Ray2012}
N.~Ray, T.~van Noorden, F.~Frank, and P.~Knabner.
\newblock Multiscale modeling of colloid and fluid dynamics in porous media
  including an evolving microstructure.
\newblock {\em Transport in Porous Media}, 95(3):669--696, 2012.

\bibitem{Schmuck2012}
M.~Schmuck.
\newblock First error bounds for the porous media approximation of the
  {P}oisson-{N}ernst-{P}lanck equations.
\newblock {\em {ZAMM} - Journal of Applied Mathematics and Mechanics /
  Zeitschrift für Angewandte Mathematik und Mechanik}, 92(4):304--319, 2012.

\bibitem{Schmuck2013}
M.~Schmuck.
\newblock New porous medium {P}oisson-{N}ernst-{P}lanck equations for strongly
  oscillating electric potentials.
\newblock {\em Journal of Mathematical Physics}, 54(2):021504, 2013.

\bibitem{Schmuck2017}
M.~Schmuck and S.~Kalliadasis.
\newblock Rate of convergence of general phase field equations in strongly
  heterogeneous media toward their homogenized limit.
\newblock {\em {SIAM} Journal on Applied Mathematics}, 77(4):1471--1492, 2017.

\bibitem{SPPK12}
M.~Schmuck, M.~Pradas, G.~A. Pavliotis, and S.~Kalliadasis.
\newblock Upscaled phase-field model for interfacial dynamics in strongly
  heterogeneous domains.
\newblock {\em Proceedings of the Royal Society A}, 468:3705--3724, 2012.

\bibitem{SSV17}
C.~Schumacher, F.~Schwarzenberger, and I.~Veseli\'c.
\newblock A {G}livenko--{C}antelli theorem for almost additive functions on
  lattices.
\newblock {\em Stochastic Processes and their Applications}, 127(1):179--208,
  2017.

\bibitem{Slodika2001}
M.~Slodi{\v{c}}ka.
\newblock Error estimates of an efficient linearization scheme for a nonlinear
  elliptic problem with a nonlocal boundary condition.
\newblock {\em {ESAIM}: Mathematical Modelling and Numerical Analysis},
  35(4):691--711, 2001.

\bibitem{Suslina2013}
T.~A. Suslina.
\newblock Homogenization of the {D}irichlet problem for elliptic systems:
  ${L}_{2}$-operator error estimates.
\newblock {\em Mathematika}, 59(2):463--476, 2013.

\bibitem{Triantafyllidis1996}
N.~Triantafyllidis and S.~Bardenhagen.
\newblock The influence of scale size on the stability of periodic solids and
  the role of associated higher order gradient continuum models.
\newblock {\em Journal of the Mechanics and Physics of Solids},
  44(11):1891--1928, 1996.

\bibitem{Sarkis}
H.~M. Versieux and M.~Sarkis.
\newblock Numerical boundary corrector for elliptic equations with rapidly
  oscillating periodic coefficients.
\newblock {\em International Journal for Numerical Methods in Biomedical
  Engineering}, 22:577--589, 2006.

\bibitem{Zhikov2016}
V.~V. Zhikov and S.~E. Pastukhova.
\newblock Operator estimates in homogenization theory.
\newblock {\em Russian Mathematical Surveys}, 71(3):417--511, 2016.

\end{thebibliography}
	\end{center}

	\medskip
	\medskip
	
\end{document}